\documentclass[12pt,a4paper]{article}
 
\usepackage[latin1]{inputenc}
\usepackage[english]{babel}
\usepackage[T1]{fontenc}
\usepackage{amsmath} 
\usepackage{amsfonts}
\usepackage{amssymb}
\usepackage{amsthm}
\usepackage{ dsfont }
\usepackage{ stmaryrd }
\usepackage{tikz}
\usetikzlibrary{matrix,arrows.meta}
\usepackage[title]{appendix}
\usepackage[normalem]{ulem} 

\allowdisplaybreaks

\usepackage{hyperref}

\usepackage{xcolor} 
\usepackage[a4paper,top=2.1cm,bottom=2.60cm,left=2.1cm,right=2.1cm]{geometry} 

\theoremstyle{plain}
\newtheorem{theorem}{Theorem}[section]
\newtheorem{lemma}[theorem]{Lemma}
\newtheorem{proposition}[theorem]{Proposition}

\theoremstyle{definition} 

\newtheorem{definition}[theorem]{Definition}

\begin{document}

\author{Matteo Dalla Riva  \thanks{Dipartimento di Ingegneria, Universit\`a degli Studi di Palermo, Viale delle Scienze, Ed. 8, 90128 Palermo, Italy.} ,  Paolo Luzzini \thanks{Dipartimento di Matematica `Tullio Levi Civita', Universit\`a degli Studi di Padova, Via Trieste 63, 35121 Padova , Italy.}}

\title{Dependence of the layer heat potentials upon support perturbations}
\maketitle

\abstract{We prove that the integral operators associated with the 
 layer heat potentials depend smoothly upon a parametrization of the support of integration. The analysis is carried out in the optimal H\"older setting.}  
 
\vspace{10pt}

\vspace{9pt}


\section{Introduction}
Potential theory is a powerful tool for analyzing boundary value problems for partial differential equations, particularly in the case of elliptic and parabolic operators. Along with existence, uniqueness, and regularity issues, potential-theoretic methods can be used to study perturbation problems. For example, potential-theoretic techniques can be employed to analyze how a solution depends on a deformation of the domain. To apply this kind of approach, it is important to understand how layer potentials and other potential-type integral operators depend on variations of the support of integration.

In the literature, most of the research in this direction focuses on the elliptic case. We mention, for example,  Coifman and Meyer's  results \cite{CoMe83} and Wu's results \cite{Wu93} on the analytic dependence of the Cauchy integral on a variable arc-length parametrized curve. Moreover, Potthast's work \cite{Po96I, Po94, Po96II} and Potthast and Stratis' results \cite{PoSt03} establish a Fr\'echet differentiability result for layer potentials associated with the Helmholtz operator and apply these findings to some inverse problems in acoustic and electromagnetic scattering. Lanza de Cristoforis and Preciso \cite{LaPr99} showed that the Cauchy integral depends analytically on a parametrization of the support. Later, Lanza de Cristoforis and Rossi \cite{LaRo04, LaRo08} considered the case of layer potentials associated with the Laplace and Helmholtz operators and proved real analyticity results which were  used in \cite{La06, La07} to study domain perturbation problems for the Laplace and Poisson equations. In \cite{DaLuMu21}, we have employed similar techniques for a perturbed obstacle scattering problem. The case of layer potentials associated with a general second-order elliptic operator with constant coefficients has been analyzed in \cite{DaLa10}, while the periodic and quasi-periodic cases are considered in Lanza de Cristoforis and Musolino \cite{LaMu11} and \cite{BrDaLuMu22}.  In a collaboration with Musolino and Pukhtaievych we have exploited the results on the periodic potentials for a shape sensitivity analysis of the longitudinal flow along a periodic array of cylinders and the effective conductivity of periodic composites (see, e.g., \cite{LuMuPu19, LuMu20, DaLuMuPu21}). More recently, Henr\'iquez and Schwab \cite{HeSc21} proved that the Calder\'on projector of the Laplacian in $\mathbb{R}^2$ is an holomorphic function of the shape of the support, a result that has been extended to $\mathbb{R}^n$ in \cite{DaLuMu22}.

Moreover, several authors have explored elliptic domain perturbation problems using approaches other than potential theory. Examples include Bucur and Buttazzo \cite{BuBu05}, Daners \cite{Da08}, Delfour and Zol\'esio \cite{DeZo11}, Henrot and Pierre \cite{HePi05}, Novotny and Soko\l owski \cite{NoSo13}, Pironneau \cite{Pi84}, and Soko\l owski and Zol\'esio \cite{SoZo92}.

So, we can conclude that the literature on shape stability and regularity is fairly rich in the case of elliptic operators. The parabolic case, instead, seems to be far less understood. Daners \cite{Da96} studied the stability of solutions of parabolic problems upon domain perturbation, while Chapko, Kress, and Yoon \cite{ChKrYo98, ChKrYo99} proved shape differentiability results for the solutions of the Dirichlet and Neumann problems for the heat equation. Then, they applied these results to some inverse problems in heat conduction. However, to the best of our knowledge, higher regularity results are still unavailable. Moreover, an analysis of the shape dependence of the integral operators arising in parabolic potential theory is lacking in the literature. To fill this gap, this paper aims to develop a high regularity theory by analyzing the dependence of layer heat potentials on variations in the shape of the support of integration. To wit, we fix a regular  open 
subset $\Omega$ of $\mathbb{R}^n$, which plays the role of a reference set,  and we
introduce a class of diffeomorphisms $\mathcal{A}_{\partial \Omega}$ from $\partial \Omega$ to $\mathbb{R}^n$ 
 (see Definition \ref{Asets} and Figure \ref{fig:phi}). 
Our main results prove that the maps that take a function $\phi\in \mathcal{A}_{\partial \Omega}$ to certain layer heat potentials supported on $\phi(\partial\Omega)$ are of class $C^\infty$.

We now describe these results  with a few more details. 
Let $S_n: \mathbb{R}^{1+n} \setminus \{0,0\} \to \mathbb{R}$ denote the usual fundamental solution of the heat operator, that is 
\begin{equation*} 
S_{n}(t,x):=
\begin{cases}
 \frac{1}{(4\pi t)^{\frac{n}{2}} }e^{-\frac{|x|^{2}}{4t}}&\quad{\mathrm{if}}\ (t,x) \in  \mathopen]0,+\infty\mathclose[  \times \mathbb{R}^n\,, 
 \\
 0 &\quad{\mathrm{if}}\  (t,x) \in  \left(\mathopen]-\infty,0] \times \mathbb{R}^n\right)\setminus \{(0,0)\}\,.
\end{cases}
\end{equation*}
Throughout the paper we find convenient to adopt the following notation: if $D$ is a subset of $\mathbb{R}^n$,
$T \in \mathopen]0,+\infty[$ and $h$ is a map from $D$ to $\mathbb{R}^n$, we denote by $h^T$ the map from  $[0,T] \times D$
 to  $[0,T] \times \mathbb{R}^n$ defined by 
\[
h^T(t,x) := (t, h(x)) \qquad \forall (t,x) \in  [0,T] \times D.
\]
Let $\phi \in \mathcal{A}_{\partial \Omega}$ and let $\mu$ be a continuous function from 
$ [0,T]  \times \partial\Omega$ to $\mathbb{R}$.  In order to work with a space of densities that is not $\phi$-dependent,
it makes sense to consider the single layer potential with density $\mu \circ (\phi^T)^{(-1)}$: 
 \begin{align*}
 v\big[\mu\, \circ\, &(\phi^T)^{(-1)}\big](t,x)\\
  :=& \int_0^t\int_{\phi(\partial\Omega)} S_n(t-\tau,x-y) \mu \circ (\phi^T)^{(-1)} (\tau,y) \,d\sigma_yd\tau\\
 &= \int_0^t\int_{\phi(\partial\Omega)} S_n(t-\tau,x-y) \mu(\tau,\phi^{(-1)}(y)) \,d\sigma_yd\tau \qquad \forall 
 (t,x) \in [0,T] \times \mathbb{R}^n.
 \end{align*}
Moreover, in the applications it is often convenient to consider the boundary integral operator associated with the $\phi$-pullback of the single layer potential, that is
 \begin{equation}\label{Vdef}
 V_{\phi}[\mu] :=  v\big[\mu \circ (\phi^T)^{(-1)}\big] \circ \phi^T \qquad \mbox{ on } [0,T] \times \partial\Omega.
 \end{equation}
We also consider the maps 
 \begin{align}\label{Vldef}
& V_{l,\phi}[\mu](t,x) := \int_0^t\int_{\phi(\partial\Omega)} \partial_{x_l}S_n(t-\tau,\phi(x)-y) \mu \circ (\phi^T)^{(-1)} (\tau,y) \,d\sigma_yd\tau,\\ \label{V*def}
  &W_{*,\phi}[\mu](t,x) :=\int_0^t\int_{\phi(\partial\Omega)}D_xS_n(t-\tau,\phi(x)-y)\cdot \nu_{\phi}(x) \mu \circ (\phi^T)^{(-1)} (\tau,y) \,d\sigma_yd\tau,\\ \label{Wdef}
   &W_\phi[\mu](t,x) := -\int_0^t\int_{\phi(\partial\Omega)}D_xS_n(t-\tau,\phi(x)-y)\cdot \nu_{\phi}(y) \mu \circ (\phi^T)^{(-1)} (\tau,y) \,d\sigma_yd\tau,
 \end{align}
 for all $(t,x) \in  [0,T] \times \partial\Omega$  and all $l\in\{1,\dots,n\}$. Here above, $ \partial_{x_l}S_n$ denotes the $x_l$-derivate of $S_n$    and $D_xS_n$ is the gradient 
 of $S_n$ with respect to the spatial variables, whereas $\nu_{\phi}$ denotes the exterior unit normal field to $\phi(\partial\Omega)$. The maps $V_{l,\phi}$ and $W_{*,\phi}$ are associated with the 
  $\phi$-pullback of the $x_l$ and normal derivatives of the single layer potential, respectively, and the map $W_\phi$ is associated with the  $\phi$-pullback of the double layer potential.  The main result of this paper states that the maps from a certain subset of $\mathcal{A}_{\partial \Omega}$ to suitable spaces of operators, which take $\phi$ to $V_\phi$, $V_{l,\phi}$, $W_{*,\phi}$, and $W_\phi$, belong to the class $C^\infty$ (see Theorem \ref{thm:main}).
 
 Our strategy to prove the result proceeds as follows. First, we characterize layer potentials as solutions of a transmission problem on a $\phi$-dependent domain. We then pullback this problem to a fixed domain and obtain a new abstract $\phi$-dependent problem on a fixed set. Finally, we use the  real analyticity of the inversion map to establish the regularity upon shape perturbations. It's worth noting that the non-homogeneous term in the heat equation for the transmission problem includes a distributional term of the form $\partial_t f$, as described in Theorem \ref{auxtranpb} and Lemma \ref{eqhe1}. Standard results in parabolic theory do not cover this case, hence we relied on new results on the heat volume potential that the second author developed for this purpose and presented in \cite{Lu21}. We acknowledge that a direct change of variable, as described in Chapko, Kress and Yoon \cite{ChKrYo98, ChKrYo99}, can avoid the issue of $\partial_t f$ in the equation. However, their approach requires the boundary to be $C^2$, and it does not work in the optimal H\"older setting, i.e. with  sets of class $C^{1,\alpha}$.  In contrast, our approach maintains sharp assumptions on the regularity of sets and perturbations.

The strategy we have just described stems from the approach proposed by Lanza de Cristoforis and Rossi 
 \cite{LaRo04, LaRo08}  to analyze the shape dependence of layer potentials for the Laplace and Helmoltz operators. 
The extension to the heat equation is, however, by no means  straightforward. On the contrary, we have to be very careful when dealing with the space-time anisotropy and we have to identify a suitable functional setting with 
 the correct regularity for the space and the time components.

 { We also observe that, instead of the real and complex analyticity results that are typical in the elliptic case (see, e.g., \cite{DaLa10} and \cite{DaLuMu22}) we have to content ourselves with a smoothness result. This seems to be an unavoidable issue for the heat equation, and its origin can be traced to the regularity of the fundamental solution  $S_n$. Specifically, $S_n$ is real analytic in $x$ for a fixed $t\neq 0$, but only $C^\infty$ in $t$ for a fixed $x\neq 0$. 

  In some  forthcoming papers, we plan to use the results presented here to study the shape sensitivity of the solutions to linear and nonlinear boundary value problems for the heat equation.
 
The paper is organized as follows.  In Section \ref{sec:notation} we introduce some notation and present some known preliminary results: in Section \ref{sec:standard} we fix some standard notation and definitions, in Section \ref{sec:spaces} we introduce certain parabolic Shauder spaces, in Section \ref{sec:diffeomors} we define the class of diffeomorphisms $\mathcal{A}_{\partial\Omega}$ and present some related results, and, finally, in Section \ref{sec:potentials} we recall the definition of layer heat potentials and their basic properties. In Section \ref{sec:trans} we prove the unique solvability of a certain auxiliary transmission problem. This is the problem that later we use to characterize the layer potentials.  Then, in Section \ref{sec:pullback}, we $\phi$-pullback the problem to a fixed domain and we analyze the dependence of the pulled-back problem upon the shape parameter $\phi$. Finally,
Section \ref{sec:main} contains our main results on the dependence of the layer potentials upon $\phi$. In Appendix \ref{appA} we collect some technical results on composition and integral operators.

\section{Preliminaries}\label{sec:notation}
\subsection{Some standard notation}\label{sec:standard}
 For 
standard definitions of calculus in normed spaces, and in particular for the definition and properties of real analytic operators,  we refer to 
Deimling~\cite{De85}.
The inverse of an invertible function $f$ is denoted by $f^{(-1)}$, while the reciprocal of a complex-valued function $g$ is denoted by $g^{-1}$.
If  $A$ is a matrix,  then $A^\top$ denotes the transpose matrix of $A$ and $A_{ij}$ denotes the $(i,j)$-entry of
$A$. If $A$ is invertible, $A^{-1}$ is the inverse of $A$ and we set $A^{-\top}:= (A^{-1})^\top$.

The symbol $\mathbb{N}$ denotes the set of natural numbers including $0$. Throughout the paper,
\[
n \in \mathbb{N} \setminus \{0,1\}
\]
denotes the dimension of the Euclidean ambient space $\mathbb{R}^n$.
If
${\mathbb{D}}\subseteq {\mathbb {R}}^{n}$, then $\overline{\mathbb{D}}$ 
denotes the 
closure of ${\mathbb{D}}$ and $\partial{\mathbb{D}}$ denotes the boundary of ${\mathbb{D}}$. 
If $\mathcal{X}$ is a normed space, then 
$B({\mathbb{D}}, {\mathcal{X}})$ and $C^{0}({\mathbb{D}}, {\mathcal{X}})$
denote the space of bounded and continuous functions from ${\mathbb{D}}$ to ${\mathcal{X}}$, respectively. As usual, we equip $B({\mathbb{D}}, {\mathcal{X}})$ with the sup-norm and we set $C^{0}_{b}({\mathbb{D}}, {\mathcal{X}})
:= C^{0}({\mathbb{D}}, {\mathcal{X}})\cap B({\mathbb{D}}, {\mathcal{X}})$.   If $\mathcal{Y}$ is also a normed space, than $\mathcal{L}(\mathcal{X},\mathcal{Y})$ denotes the space of bounded linear operators from $\mathcal{X}$ to $\mathcal{Y}$ equipped with the usual operator norm.

 Let $\Omega$ be an open 
subset of ${\mathbb{R}}^{n}$.  Then $\Omega^- := \mathbb{R}^n\setminus \overline \Omega$
denotes the exterior of $\Omega$.
Let $m \in \mathbb{N}$. 
The space of $m$-times continuously 
differentiable real-valued functions on $\Omega$ is denoted by 
$C^{m}(\Omega)$. 
Let $f\in  C^{m}(\Omega)  $. Then   $Df$ denotes the Jacobian matrix of $f$. 
The
subspace of $C^{m}(\Omega )$ of those functions $f$ whose derivatives $D^{\eta }f$ of
order $|\eta |\leq m$ can be extended with continuity to 
$\overline\Omega$  is  denoted $C^{m}(
\overline\Omega )$.
The subspace of $C^{m}(\overline\Omega ) $  whose derivatives up to the $m$ order are bounded is denoted 
$C^{m}_{b}(\overline\Omega ) $. As is well known,  $C^{m}_{b}(\overline\Omega )$ 
equipped with the norm
$\|f\|_{C^{m}_{b}(\overline\Omega )}:= \sum
_{|\eta |\leq m}\sup_{\overline\Omega}|D^\eta f|$ is a Banach space.  The symbol $\nu_{\Omega}$ denotes 
 the outward unit normal field to $\partial \Omega$, where it exists.

Let $m \in \mathbb{N}$ and $\alpha\in\mathopen]0,1[$. 
For the definition of open subsets of ${\mathbb{R}}^{n}$ of class $C^m$ and
$C^{m,\alpha}$,
and of the Schauder spaces $C^{m,\alpha}(\overline\Omega)$ and $C^{m, \alpha}(\partial\Omega)$, 
 we refer  to Gilbarg and Trudinger~\cite[pp. 52, 95]{GiTr83}.

 \subsection{Parabolic Schauder spaces}\label{sec:spaces}
 
  Let $\alpha, \beta\in \mathopen]0,1[$,  $T \in \mathopen ]0,+\infty[$ and $\mathbb{D} \subseteq \mathbb{R}^n$.
 Then $C^{\alpha;\beta}([0,T] \times\mathbb{D})$
denotes the space of bounded continuous functions $u$ from $[0,T] \times \mathbb{D}$ to ${\mathbb{R}}$ such that
\begin{align*}
\|u\|_{C^{\alpha;\beta}([0,T] \times\mathbb{D})}
:= &\sup_{ [0,T] \times\mathbb{D} }|u|
+\sup_{\substack{t_{1},t_{2}\in  [0,T]\\ t_{1}\neq t_{2}}   }\sup_{x \in \mathbb{D}}
\frac{|u(t_{1},x)  -u(t_{2},x)  | }{|t_{1}-t_{2}|^{\alpha}}\\ \nonumber
&+\sup_{t\in  [0,T] } 
\sup_{\substack{ x_1,x_2 \in \mathbb{D}\\ x_{1}\neq x_{2}}}
\frac{|u(t,x_1)  -u(t,x_2)  | }{|x_{1}-x_{2}|^{\beta}}<+\infty.
\end{align*}
%
%
Let now $\Omega$ be an open subset of $\mathbb{R}^n$. Then $C^{\frac{1+\alpha}{2};1+\alpha}([0,T] \times \overline{\Omega})$ denotes the space of  bounded continuous functions $u$ from $[0,T] \times \overline{\Omega}$ to ${\mathbb{R}}$ which are continuously differentiable with respect to the space variables and such that 
\begin{align*}
\|u\|_{C^{\frac{1+\alpha}{2};1+\alpha}([0,T] \times \overline{\Omega})}
:= &\sup_{  [0,T] \times  \overline{\Omega}}|u| + \sum_{i=1}^n\|\partial_{x_i}u\|_{C^{\frac{\alpha}{2};\alpha}([0,T] \times \overline{\Omega})}\\
&+\sup_{\substack{t_{1},t_{2}\in  [0,T]\\ t_{1}\neq t_{2}}   }\sup_{x \in \overline{\Omega}}
\frac{|u(t_{1},x)  -u(t_{2},x)  | }{|t_{1}-t_{2}|^{\frac{1+\alpha}{2}}}<+\infty.
\end{align*}
If $\Omega$ is of class $C^{1,\alpha}$, we can use the  local parametrization of $\partial\Omega$ to 
define the space 
$C^{\frac{1+\alpha}{2};1+\alpha}([0,T] \times \partial\Omega)$ in the natural way.
Similarly, we can  define the spaces $C^{m,\alpha}(M)$ and $C^{\frac{m+\alpha}{2};m+\alpha}([0,T] \times M)$, 
$m=0,1$ on a 
 manifold $M$ of class $C^{m,\alpha}$ imbedded in 
$\mathbb{R}^n$  (for details, see Appendix \ref{appA}).

 Finally, we use the subscript $0$ to denote a subspace consisting of functions that are zero at  $t=0$. For example,
\[
C_0^{\alpha;\beta}([0,T] \times \mathbb{D}) := \left\{u \in C^{\alpha;\beta}([0,T] \times  \mathbb{D}) : 
u(0,x)=0 \quad   \forall x \in \mathbb{D}\right\}.
\]
Then the spaces  $C_0^{\frac{1+\alpha}{2};1+\alpha}([0,T] \times \overline{\Omega})$,   $C_0^{\frac{1+\alpha}{2};1+\alpha}([0,T] \times \partial \Omega)$, 
and  $C_0^{\frac{m+\alpha}{2};m+\alpha}([0,T] \times M)$ are similarly defined. 
 
 For a comprehensive introduction to parabolic 
Schauder spaces we refer the reader to classical monographs on the field, for example   Lady\v{z}enskaja, Solonnikov and Ural'ceva 
\cite[Ch. 1]{LaSoUr68}  (see also  \cite{LaLu17, LaLu19}).

\subsection{The class of diffeomorphisms $\mathcal{A}_{\partial\Omega}$}\label{sec:diffeomors}

We now introduce the class of diffeomorphisms that we use to model the domains' shape. 
\begin{definition}\label{Asets}
Let 
$\Omega$ be a bounded open connected subset of $\mathbb{R}^n$ of class $C^{1}$. We denote by  $\mathcal{A}_{\partial \Omega}$  the set consisting of the functions of class $C^1(\partial\Omega, \mathbb{R}^{n})$ that are injective and whose differential is injective at all points  of $\partial\Omega$ and, similarly, we denote by $\mathcal{A}_{\overline{\Omega}}$  the set of functions of  $C^1(\overline{\Omega}, \mathbb{R}^n)$ that are injective and whose differential is injective at all points  of  $ {\overline{\Omega}}$. 
\end{definition}
\begin{figure}
\center
\includegraphics[height=2.6cm]{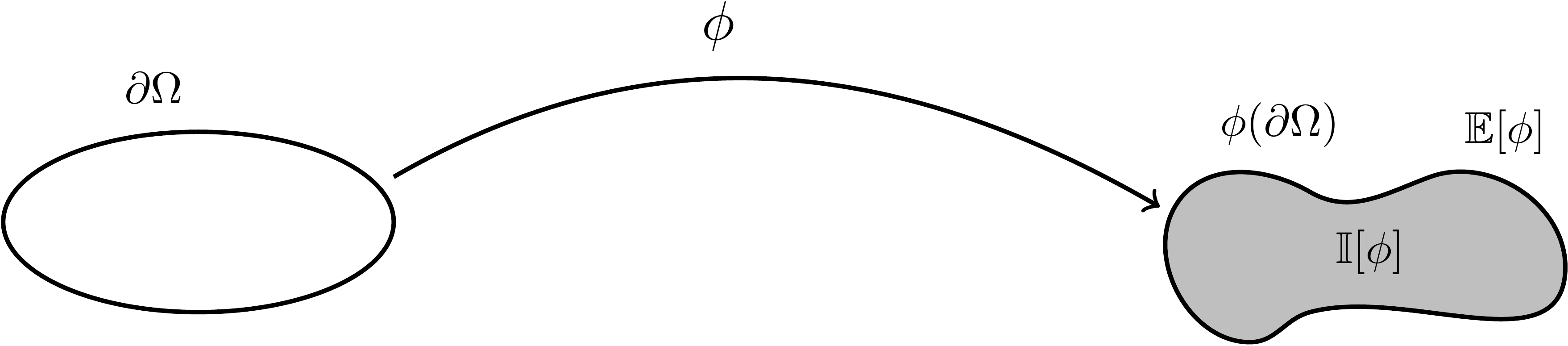}
%
%
%
%
%
%
%
\caption{{\it The diffeomorphism  $\phi$ and the $\phi$-dependent sets $\mathbb{I}[\phi]$,  $\mathbb{E}[\phi]$  and  $\phi(\partial\Omega)$.}}\label{fig:phi}
\end{figure} 

We can verify that $\mathcal{A}_{\partial \Omega}$ and $\mathcal{A}_{\overline{\Omega}}$ are open 
in $ C^1(\partial\Omega, \mathbb{R}^{n})$ and $C^1({\overline{\Omega}, \mathbb{R}^n})$, respectively
(see,  e.g., Lanza de Cristoforis and Rossi \cite[Lem.~2.2, p.~197]{LaRo08}  
and \cite[Lem.~2.5, p.~143]{LaRo04}). If $\Omega$ has connected exterior $\Omega^-$, then $\mathbb{R}^n\setminus\partial\Omega$ has two open connected components and thus  the Jordan-Leray separation theorem ensures that 
$\mathbb{R}^n \setminus \phi(\partial\Omega)$ has exactly two open connected components  for all $\phi \in \mathcal{A}_{\partial \Omega}$ {(see, e.g.,   Deimling \cite[Thm 5.2, p.26]{De85}).  One of these open connected components is bounded, and we denote it by $\mathbb{I}[\phi]$, the letter ``$\mathbb{I}$'' standing for ``interior.'' The other one is unbounded, and we denote it by $\mathbb{E}[\phi]$, the letter ``$\mathbb{E}$'' standing for ``exterior.'' (See Figure \ref{fig:phi}.) 

We need to recall two technical lemmas of Lanza de Cristoforis and Rossi \cite[\S 2]{LaRo08}, which show that a diffeomorphism on $\partial\Omega$ can be extended in a neighborhood of $\partial\Omega$ by means of a real analytic extension operator.
\begin{lemma}\label{ext1}
Let $\alpha \in \mathopen]0,1[$.
 Let $\Omega$ be a bounded open connected subset  of   $\mathbb{R}^n$  of class $C^{1,\alpha}$  such that $\Omega^-$ is connected. 
 There exists $\omega \in C^{1,\alpha}(\partial\Omega, \mathbb{R}^n)$ such that $|\omega|=1$  and 
$\omega \cdot \nu_\Omega>1/2$ on $\partial \Omega$. Moreover, the 
following statements hold.
\begin{itemize}
\item[(i)] There exists $\delta_\Omega \in \mathopen]0,+\infty[$ such that the sets
\begin{align*}
&\Omega_{\omega,\delta} := \{x+s\omega(x) : x \in \partial\Omega,\, s \in \mathopen]-\delta,\delta[\},\\
&\Omega_{\omega,\delta}^+ := \{x+s\omega(x) : x \in \partial\Omega,\, s \in \mathopen]-\delta,0[\},\\
&\Omega_{\omega,\delta}^- := \{x+s\omega(x) : x \in \partial\Omega,\, s \in \mathopen]0,\delta[\}
\end{align*} 
are connected and of class $C^{1,\alpha}$. They have boundaries 
\begin{align*}
&\partial\Omega_{\omega,\delta} = \{x+s\omega(x) : x \in \partial\Omega,\, s \in \{-\delta,\delta\}\},\\
&\partial\Omega_{\omega,\delta}^+ = \{x+s\omega(x) : x \in \partial\Omega,\, s \in \{-\delta,0\}\},\\
&\partial\Omega_{\omega,\delta}^- = \{x+s\omega(x) : x \in \partial\Omega,\, s \in \{0,\delta\}\}.
\end{align*} 
and we have $\Omega_{\omega,\delta}^+\subseteq \Omega$ and $ \Omega_{\omega,\delta}^-\subseteq \Omega^-$ for all $\delta \in\mathopen ]0,\delta_\Omega[$.
\item[(ii)] Let $\delta \in \mathopen]0,\delta_\Omega[$. If $\Phi \in \mathcal{A}_{\overline{\Omega_{\omega,\delta}}}$, then $\phi := \Phi_{|\partial\Omega} \in \mathcal{A}_{\partial\Omega}$.
\item[(iii)] If  $\delta \in\mathopen ]0,\delta_\Omega[$, then the set
	\[
 	    \mathcal{A}'_{\overline{\Omega_{\omega,\delta}}} := \{\Phi \in  \mathcal{A}_{\overline{\Omega_{\omega,\delta}}} : 
		\Phi(\Omega_{\omega,\delta}^+) \subseteq \mathbb{I}[\Phi_{|\partial\Omega}]\}
	\]
	is open in $\mathcal{A}_{\overline{\Omega_{\omega,\delta}}}$ and $\Phi(\Omega_{\omega,\delta}^-) \subseteq \mathbb{E}[\Phi_{|\partial\Omega}]$ for all
	$\Phi \in  \mathcal{A}'_{\overline{\Omega_{\omega,\delta}}} $.
\item[(iv)]  If  $\delta \in \mathopen]0,\delta_\Omega[$ and $\Phi \in  C^{1,\alpha}(\overline{\Omega_{\omega,\delta}}, \mathbb{R}^n)\cap\mathcal{A}'_{\overline{\Omega_{\omega,\delta}}}$,
	 then both  $\Phi(\Omega_{\omega,\delta}^+) $ and $\Phi(\Omega_{\omega,\delta}^-)$ are open sets of class $C^{1,\alpha}$, and 
	\[
	  \partial \Phi(\Omega_{\omega,\delta}^+) = \Phi(\partial\Omega_{\omega,\delta}^+), \quad 
 	\partial \Phi(\Omega_{\omega,\delta}^-) = \Phi(\partial\Omega_{\omega,\delta}^-).
	\]
\end{itemize}
\end{lemma}

\begin{lemma}\label{ext2}
Let  $\Omega$, $\omega$, $\delta_\Omega$ be as in Lemma \ref{ext1}. Let $\phi_0\in  C^{1,\alpha}(\partial\Omega, \mathbb{R}^n)\cap\mathcal{A}_{\partial\Omega}$. 
 Then the following statements hold.
 \begin{itemize}
\item[(i)] There exists $\delta_0 \in \mathopen]0,\delta_\Omega[$ and $\Phi_0 \in  C^{1,\alpha}(\overline{\Omega_{\omega,\delta_0}}, \mathbb{R}^n) \cap\mathcal{A}'_{\overline{\Omega_{\omega,\delta_0}}}  $ such that $\phi_0 = {\Phi_0}_{|\partial\Omega}$.
\item[(ii)] Let  $\delta_0 $,  $\Phi_0$ be as in statement $(i)$. Then there exist an open neighborhood $\mathcal{W}_0$ of $\phi_0$ in $\ C^{1,\alpha}(\partial\Omega, \mathbb{R}^n)\cap\mathcal{A}_{\partial\Omega} $, and a real analytic extension operator 
	 $\mathbf{E}$ from $C^{1,\alpha}(\partial\Omega, \mathbb{R}^n)$ to $C^{1,\alpha}(\overline{\Omega_{\omega, \delta_0}}, \mathbb{R}^n)$ which 
	maps  $\mathcal{W}_0$ to  $C^{1,\alpha}(\overline{\Omega_{\omega,\delta_0}}, \mathbb{R}^n) \cap\mathcal{A}'_{\overline{\Omega_{\omega,\delta_0}}} $ 
	and such that  $\mathbf{E}[\phi_0] = \Phi_0$ and $\mathbf{E}[\phi]_{|\partial\Omega} = \phi$, for all $\phi \in \mathcal{W}_0$.  
\end{itemize}
\end{lemma}

Finally, we present in the following technical lemma two real analyticity results, one for a map related to the change of variables 
in integrals, and one for the pullback of the outer normal field.
For a proof we refer to Lanza de Cristoforis and Rossi \cite[p.~166]{LaRo04}
and to Lanza de Cristoforis \cite[Prop. 1]{La07}. Throughout the paper $\nu_{\phi}$ denotes the exterior unit normal field to $\partial \mathbb{I}[\phi]=\phi(\partial\Omega)$.
\begin{lemma}\label{rajacon}
Let  $\Omega$ be  as in Lemma \ref{ext1}.   Then the following statements hold.
\begin{itemize}
\item[(i)] For each $\phi \in C^{1,\alpha}(\partial\Omega, \mathbb{R}^{n})\cap\mathcal{A}_{\partial \Omega}$ there exists a unique  
$\tilde \sigma_n[\phi] \in C^{1,\alpha}(\partial\Omega)$ such that 
\[ 
  \int_{\phi(\partial\Omega)}f(s)\,d\sigma_s=  \int_{\partial\Omega}f \circ \phi(y)\tilde\sigma_n[\phi](y)\,d\sigma_y \qquad \forall f \in L^1(\phi(\partial\Omega)).
\]
Moreover,  $\tilde \sigma_n[\phi]>0$ and the map that takes $\phi$ to $\tilde \sigma_n[\phi]$  is real analytic from $C^{1,\alpha}(\partial\Omega, \mathbb{R}^n)\cap\mathcal{A}_{\partial \Omega}  $ to $ C^{0,\alpha}(\partial\Omega)$.
\item[(ii)] The map from $C^{1,\alpha}(\partial\Omega, \mathbb{R}^n)\cap\mathcal{A}_{\partial \Omega} $ to $ C^{0,\alpha}(\partial\Omega, \mathbb{R}^{n})$ that takes $\phi$ to $\nu_{\phi} \circ \phi$ is real analytic.
\end{itemize}
\end{lemma}

\subsection{Layer heat potentials} \label{sec:potentials}

In this section we collect some well-known facts on the  layer heat potentials. For proofs and detailed references  
we refer to    Lady\v{z}enskaja, Solonnikov and Ural'ceva 
\cite{LaSoUr68}  and \cite{LaLu19, Lu19}. In Theorem \ref{pdhp} we introduce the double layer potential and describe some of its main properties.

\begin{theorem}\label{pdhp} 
Let $\alpha \in \mathopen]0, 1[$. Let $T \in \mathopen]0, +\infty[$. Let $\Omega$ be a bounded open subset of 
$\mathbb{R}^n$ of class $C^{1,\alpha}$.  The double layer potential with density $\mu \in C^{\frac{\alpha}{2};\alpha}_0([0,T] \times \partial\Omega)$ is the function $w[\mu]$ from $[0,T] \times \mathbb{R}^n$ to $\mathbb{R}$ defined by
\begin{equation*}\label{pdhp1}
w[ \mu](t, x) := \int_{0}^t \int_{\partial \Omega} \frac{\partial }{\partial \nu_{\Omega}(y)}
S_n (t-\tau, x-y)\mu(\tau, y) \,d\sigma_y d\tau \qquad \forall (t, x) \in [0,T]\times \mathbb{R}^n\,.
\end{equation*}
For the double layer potential the following statements hold.
\begin{itemize}
\item[(i)] $w[\mu]$ solves the heat equation in $[0,T] \times (\mathbb{R}^n \setminus \partial \Omega)$.
\item[(ii)] The 
restriction $w[\mu]_{|[0,T] \times \Omega}$ has a unique extension to a continuous function 
$w^+[\mu]$ from $[0,T] \times \overline{\Omega}$ to $\mathbb{R}$ and the restriction 
$w[\mu]_{|[0,T] \times \Omega^-}$ has a unique extension to a continuous function 
$w^-[\mu]$ from $[0,T] \times \overline{\Omega^-}$ to $\mathbb{R}$.
\item[(iii)] Unless $\mu=0$, $w[\mu]$ is not continuous on the boundary $[0,T] \times \partial\Omega$ and  we have the  jump formula
\begin{align*}\label{pdhp2}
w^{\pm}[\mu] = \mp \frac{1}{2}\mu+
w[\mu]  \qquad\mbox{ on } [0,T] \times \partial\Omega.
 \end{align*}
 In addition,  for the normal derivative of $w[\mu]$ we have
 \begin{align*}
 \frac{\partial}{\partial \nu_{\Omega}}w^{+}[\mu] =
  \frac{\partial}{\partial \nu_{\Omega}}w^{-}[ \mu] \qquad\mbox{ on }[0,T] \times \partial\Omega.
 \end{align*}
\item[(iv)]  
The map from 
$C^{\frac{1+\alpha}{2} ; 1+\alpha}_0([0,T] \times \partial\Omega)$ to $C^{\frac{1+\alpha}{2} ; 1+\alpha}_0\big([0,T] \times \overline{\Omega}\big)$ 
that takes $\mu$ to $w^+[\mu]$ is linear and continuous. If $R >0$ is such that $\overline{\Omega} \subseteq \mathbb{B}_n(0,R)$,  then 
the map from 
$C^{\frac{1+\alpha}{2} ; 1+\alpha}_0([0,T] \times \partial\Omega)$ to $C^{\frac{1+\alpha}{2} ; 1+\alpha}_0\big([0,T]\times (\overline{\mathbb{B}_n(0,R)} \setminus \Omega)\big)$ 
that takes $\mu$ to $w^-[\mu]_{|[0,T] \times (\overline{\mathbb{B}_n(0,R)} \setminus \Omega)} $ is linear and continuous.
\end{itemize}
\end{theorem}

In the following Theorem \ref{pshp} we introduce the single layer potential.

\begin{theorem}\label{pshp}
Let $\alpha \in \mathopen]0, 1[$. Let $T \in \mathopen]0, +\infty[$. Let $\Omega$ be a bounded open subset of 
$\mathbb{R}^n$ of class $C^{1,\alpha}$.   The single layer potential with density $\mu \in C^{\frac{\alpha}{2};\alpha}_0([0,T] \times\partial \Omega)$ is the function $v[\mu]$ from $[0,T]\times \mathbb{R}^n$ to $\mathbb{R}$  defined by
\begin{equation*}\label{pshp1}
v[\mu](t, x) := \int_{0}^t \int_{\partial \Omega}
S_n (t-\tau, x-y)\mu(\tau, y)\, d\sigma_y d\tau \qquad \forall (t, x) \in [0,T]\times \mathbb{R}^n\,.
\end{equation*}
For the single layer potential the following statements hold.
\begin{itemize}
\item[(i)] $v[\mu]$ solves the heat equation in $ [0,T]\times(\mathbb{R}^n \setminus \partial \Omega)$.
\item[(ii)]  $v[\mu]$ is continuous in 
$[0,T]\times\mathbb{R}^n$. 
\item[(iii)]  Let  $v^+[\mu]$ and $v^-[\mu]$ denote the restriction of 
$v[\mu]$ 
to $[0,T]\times\overline{\Omega}$ and to $[0,T]\times\overline{\Omega^-}$.
Then the following jump relations hold
\begin{align*} 
\frac{\partial}{\partial \nu_{\Omega} (x)} v^\pm[\mu](t,x) = 
\pm \frac{1}{2}\mu(t,x) 
+ \int_{0}^t \int_{\partial \Omega} \frac{\partial }{\partial \nu_{\Omega}(x)}
 S_n (t-\tau, x-y)\mu(\tau, y)\, d\sigma_y d\tau,
 \\ \nonumber
 \frac{\partial}{\partial x_r} v^\pm[\mu](t,x) = 
\pm \frac{1}{2}\mu(t,x)(\nu_{\Omega})_r(x) +
 \int_{0}^t \int_{\partial \Omega} \frac{\partial }{\partial x_r}
 S_n (t-\tau, x-y)\mu(\tau, y) \,d\sigma_y d\tau,
\end{align*}
for all $(t,x) \in [0,T]\times\partial \Omega$ and all $r \in \{1,\ldots,n\}$.
\item[(iv)]  The map from $C^{\frac{\alpha}{2} ; \alpha}_0([0,T] \times\partial \Omega)$ to 
$C^{\frac{1+\alpha}{2} ; 1+\alpha}_0\big([0,T]\times\overline{\Omega}\big)$ 
that takes $\mu$ to $v^+[\mu]$ is linear and continuous. If $R >0$ and $\overline{\Omega} \subseteq \mathbb{B}_n(0,R)$,  then the map from $C^{\frac{\alpha}{2} ; \alpha}_0([0,T] \times\partial \Omega)$ to 
$C^{\frac{1+\alpha}{2} ; 1+\alpha}_0\big([0,T]\times(\overline{\mathbb{B}_n(0,R)} \setminus \Omega)\big)$ 
that takes $\mu$ to $v^-[\mu]_{|[0,T]\times(\overline{\mathbb{B}_n(0,R)} \setminus \Omega)}$ is linear and continuous.

\end{itemize}
\end{theorem}

\section{An auxiliary transmission problem}\label{sec:trans}

In the following Theorem \ref{auxtranpb} we introduce an auxiliary transmission problem and prove the existence and uniqueness of its solution. We will use such problem to characterize the layer heat potentials supported on $\Phi(\partial\Omega)$.

To shorten our statements we find convenient to introduce the following notation: If $\alpha \in \mathopen]0,1[$,  $T \in \mathopen]0,+\infty[$, $\Omega$, $\omega$ and  $\delta_\Omega$ are  as in Lemma \ref{ext1},  $\delta \in \mathopen]0,\delta_\Omega[$, and  $\Phi \in C^{1,\alpha}(\overline{ \Omega_{\omega,\delta}}, \mathbb{R}^n) \cap \mathcal{A}'_{\overline{\Omega_{\omega,\delta}}}$, then  $\mathcal{S}_\Phi$ denotes the product space 
\begin{align}\label{def:Sphi}
\mathcal{S}_\Phi :=  \; &  C_0^{\frac{1+\alpha}{2};\alpha}\left([0,T]\times\overline{\Phi(\Omega^+_{\omega,\delta})}\right)      
	\times    C_0^{\frac{1+\alpha}{2};\alpha}\left([0,T]\times\overline{\Phi(\Omega^-_{\omega,\delta})}\right)\\ \nonumber
		& \times 
		C^{\frac{\alpha}{2};\alpha}\left([0,T] \times \overline{\Phi(\Omega_{\omega,\delta}^+)}, \mathbb{R}^n\right) \times  
	 C^{\frac{\alpha}{2};\alpha}\left([0,T]\times \overline{\Phi(\Omega_{\omega,\delta}^-)}, \mathbb{R}^n\right)   \\ \nonumber
	&\times
	 C^{\frac{1+\alpha}{2}; 1+\alpha}_0\big([0,T]\times\Phi(\partial\Omega)\big) 
	 \times C^{\frac{\alpha}{2}; \alpha}_0\big([0,T]\times\Phi(\partial\Omega)\big)    \\ \nonumber
	&\times C^{\frac{1+\alpha}{2}; 1+\alpha}_0\big([0,T]\times\Phi(\partial \Omega_{\omega,\delta}^+\setminus\partial\Omega)\big) \times
	C^{\frac{1+\alpha}{2}; 1+\alpha}_0\big([0,T]\times\Phi(\partial \Omega_{\omega,\delta}^-\setminus\partial\Omega)\big).
\end{align}
Moreover, we set
\[
B[\Phi](v^+,v^-) := (D v^+)_{|[0,T]\times\Phi(\partial\Omega)}\nu_{\Phi_{|\partial\Omega}} -  (D v^-)_{|[0,T]\times\Phi(\partial\Omega)}\nu_{\Phi_{|\partial\Omega}}, 
\]
for all $(v^+, v^-) \in C^{\frac{1+\alpha}{2}; 1+\alpha}_0\left([0,T]\times\overline{\Phi(\Omega_{\omega,\delta}^+)}\right) \times 
	C^{\frac{1+\alpha}{2}; 1+\alpha}_0\left([0,T]\times\overline{\Phi(\Omega_{\omega,\delta}^-)}\right)$.

We are now ready to state  Theorem \ref{auxtranpb}. 
\begin{theorem}\label{auxtranpb}
Let $\alpha \in \mathopen]0,1[$,  $T \in \mathopen]0,+\infty[$.
Let $\Omega$, $\omega$, $\delta_\Omega$ be as in Lemma \ref{ext1}. Let $\delta \in \mathopen ]0,\delta_\Omega[$ 
and $\Phi \in C^{1,\alpha}(\overline{ \Omega_{\omega,\delta}}, \mathbb{R}^n) \cap \mathcal{A}'_{\overline{\Omega_{\omega,\delta}}}$. 
Then the transmission problem
\begin{equation}\label{auxtranpb1}
\begin{cases}
\partial_t v^+ - \Delta v^+ = \partial_t f_0^+ +\mathrm{div} f_1^+ & \mbox{ in } ]0,T\mathclose[\times\Phi(\Omega_{\omega,\delta}^+), \\
\partial_t v^- - \Delta v^- = \partial_t f_0^-+ \mathrm{div}  f_1^- & \mbox{ in }]0,T\mathclose[\times\Phi(\Omega_{\omega,\delta}^-) \\
v^+-v^- = g  & \mbox{ on } [0,T]\times\Phi(\partial\Omega),\\
B[\Phi](v^+,v^-) = g_1 & \mbox{ on } [0,T]\times\Phi(\partial\Omega),\\
v^+ = h^+ & \mbox{ on }[0,T]\times \Phi(\partial\Omega_{\omega,\delta}^+ \setminus\partial\Omega),\\
v^-=h^-  & \mbox{ on } [0,T]\times\Phi(\partial\Omega_{\omega,\delta}^- \setminus\partial\Omega),\\
v^+(0,\cdot) =0  & \mbox{ in } \overline{\Phi(\Omega_{\omega,\delta}^+)},\\
v^-(0,\cdot) =0  & \mbox{ in }\overline{\Phi(\Omega_{\omega,\delta}^-)},
\end{cases}
\end{equation}
has a unique solution $(v^+, v^-)$ in  $C^{\frac{1+\alpha}{2}; 1+\alpha}_0\left([0,T]\times\overline{\Phi(\Omega_{\omega,\delta}^+)}\right) \times 
	C^{\frac{1+\alpha}{2}; 1+\alpha}_0\Big([0,T]\times\overline{\Phi(\Omega_{\omega,\delta}^-)}\Big)$ for each given 
$(f_0^+,f_0^-,f_1^+, f_1^-, g, g_1, h^+, h^-) \in \mathcal{S}_\Phi$. 
\end{theorem}
We observe that the heat equations in \eqref{auxtranpb1} are to be understood in the weak sense of 
distributions in every instance.

\begin{proof}
First we consider the existence of a solution of the transmission problem (\ref{auxtranpb1}). Let $(f_0^+,f_0^-,f_1^+, f_1^-,g, g_1, h^+, h^-) \in \mathcal{S}_\Phi$.  In particular 
$f_0^\pm \in C_0^{\frac{1+\alpha}{2};\alpha}\left([0,T] \times \overline{\Phi(\Omega^\pm_{\omega,\delta})}\right)$.
By \cite[Thm 4.4]{Lu21} there exist functions $P[\partial_t f^\pm_0] \in C_0^{\frac{1+\alpha}{2};1+\alpha}\left([0,T]\times \overline{\Phi(\Omega_{\omega,\delta}^\pm)}\right)$--which are the heat volume potentials with densities $\partial_t f_0^\pm$--that solve the equations
 \[
 \partial_t u - \Delta u = \partial_t  f_0^\pm \qquad \mbox{ in } \mathopen]0,T\mathclose[\times\Phi(\Omega_{\omega,\delta}^\pm)
 \]
 in the sense of distributions.
Then, the proof of the existence of a solution 
of problem \eqref{auxtranpb1} can be reduced to that of the existence  of a solution for
\begin{equation}\label{auxtranpb2}
\begin{cases}
\partial_t v^+ - \Delta v^+ = \mathrm{div} f_1^+  & \mbox{ in } ]0,T\mathclose[\times\Phi(\Omega_{\omega,\delta}^+), \\
\partial_t v^- - \Delta v^- = \mathrm{div} f_1^-& \mbox{ in } ]0,T\mathclose[\times\Phi(\Omega_{\omega,\delta}^-), \\
v^+-v^- = g  & \mbox{ on } [0,T]\times\Phi(\partial\Omega),\\
B[\Phi](v^+,v^-) = g_1 & \mbox{ on }[0,T]\times \Phi(\partial\Omega),\\
v^+ = h^+ & \mbox{ on } [0,T]\times\Phi(\partial\Omega_{\omega,\delta}^+ \setminus\partial\Omega),\\
v^-=h^-  & \mbox{ on } [0,T]\times\Phi(\partial\Omega_{\omega,\delta}^- \setminus\partial\Omega),\\
v^+(0,\cdot) =0  & \mbox{ in } \overline{\Phi(\Omega_{\omega,\delta}^+)},\\
v^-(0,\cdot) =0  & \mbox{ in }\overline{\Phi(\Omega_{\omega,\delta}^-)},
\end{cases}
\end{equation} 
with a possibly different array of data $(f_1^+, f_1^-,g, g_1, h^+, h^-)$.
 By known results in parabolic theory
(see  Lunardi and Vespri \cite[Thm. 4.3]{LuVe91} and 
Lieberman \cite[Thm. 6.48]{Li96}) there exists a pair  $(\tilde v^+,\tilde v^-)  \in C^{\frac{1+\alpha}{2}; 1+\alpha}_0\left([0,T]\times\overline{\Phi(\Omega_{\omega,\delta}^+)}\right) \times C^{\frac{1+\alpha}{2}; 1+\alpha}_0\left([0,T]\times\overline{\Phi(\Omega_{\omega,\delta}^-)}\right)$   such that
\begin{equation*}
\begin{cases}
\partial_t \tilde v^+ - \Delta \tilde v^+ = \mathrm{div} f_1^+ & \mbox{ in }]0,T\mathclose[\times \Phi(\Omega_{\omega,\delta}^+), \\
\tilde v^+ = g  & \mbox{ on } [0,T]\times\Phi(\partial\Omega),\\
\tilde v^+ = h^+ & \mbox{ on } [0,T]\times\Phi(\partial\Omega_{\omega,\delta}^+ \setminus\partial\Omega),\\
\tilde v^+(0,\cdot) =0  & \mbox{ in } \overline{\Phi(\Omega_{\omega,\delta}^+)},
\end{cases}
\end{equation*}
and
\begin{equation*}
\begin{cases}
\partial_t \tilde v^- - \Delta \tilde v^- = \mathrm{div} f_1^- & \mbox{ in } ]0,T\mathclose[\times\Phi(\Omega_{\omega,\delta}^-), \\
\tilde v^- = 0  & \mbox{ on }[0,T]\times \Phi(\partial\Omega),\\
\tilde v^- = h^- & \mbox{ on } [0,T]\times\Phi(\partial\Omega_{\omega,\delta}^- \setminus\partial\Omega),\\
\tilde v^-(0,\cdot) =0  & \mbox{ in }\overline{\Phi(\Omega_{\omega,\delta}^-)}.
\end{cases}
\end{equation*}
Moreover, by Theorem \ref{pshp}, the pair  of functions $(u^+,u^-) \in C^{\frac{1+\alpha}{2}; 1+\alpha}_0\left([0,T]\times\overline{\Phi(\Omega_{\omega,\delta}^+)}\right) \times 
	C^{\frac{1+\alpha}{2}; 1+\alpha}_0\left([0,T]\times\overline{\Phi(\Omega_{\omega,\delta}^-)}\right)$
	defined by 
	\[
u^\pm := v^\pm[  \mu] \quad \mbox{ in }  [0,T]\times\overline{\Phi(\Omega_{\omega,\delta}^\pm)}, 
\]
with  
\[
\mu := -g_1 + B[\Phi](\tilde v^+, \tilde v^-) \quad \mbox{ on } [0,T]\times \Phi(\partial\Omega),
\]
is a solution of the transmission problem
\begin{equation*}
\begin{cases}
\partial_t  u^+ - \Delta  u^+ = 0 & \mbox{ in } ]0,T\mathclose[\times\Phi(\Omega_{\omega,\delta}^+), \\
\partial_t  u^- - \Delta  u^- = 0 & \mbox{ in }]0,T\mathclose[\times\Phi(\Omega_{\omega,\delta}^-), \\
u^+ - u^- = 0  & \mbox{ on } [0,T]\times\Phi(\partial\Omega),\\
B[\phi](u^+,u^-) = -g_1 + B[\Phi](\tilde v^+, \tilde v^-)  & \mbox{ on }[0,T]\times \Phi(\partial\Omega),\\ 
u^+(0,\cdot) =0  & \mbox{ in } \overline{\Phi(\Omega_{\omega,\delta}^+)},\\
u^-(0,\cdot) =0  & \mbox{ in } \overline{\Phi(\Omega_{\omega,\delta}^-)}.
\end{cases}
\end{equation*}
Then,   if the boundary value problem 
\begin{equation}\label{auxtranpb3}
\begin{cases}
\partial_t  \hat v^+ - \Delta  \hat v^+ = 0 & \mbox{ in } ]0,T\mathclose[\times\Phi(\Omega_{\omega,\delta}^+), \\
\partial_t  \hat v^- - \Delta  \hat v^- = 0 & \mbox{ in } ]0,T\mathclose[\times\Phi(\Omega_{\omega,\delta}^-), \\
\hat v^+ - \hat v^- = 0  & \mbox{ on }[0,T] \times \Phi(\partial\Omega),\\
B[\phi](\hat v^+,\hat v^-) =0  & \mbox{ on }[0,T] \times \Phi(\partial\Omega),\\ 
\hat v^+ = u^+ & \mbox{ on }[0,T]\times \Phi(\partial\Omega_{\omega,\delta}^+ \setminus\partial\Omega),\\
\hat v^-=u^-  & \mbox{ on } [0,T]\times\Phi(\partial\Omega_{\omega,\delta}^- \setminus\partial\Omega),\\
\hat v^+(0,\cdot) =0  & \mbox{ in } \overline{\Phi(\Omega_{\omega,\delta}^+)},\\
\hat v^-(0,\cdot) =0  & \mbox{ in } \overline{\Phi(\Omega_{\omega,\delta}^-)},
\end{cases}
\end{equation}
has a solution $(\hat v^+,\hat v^-) \in  C^{\frac{1+\alpha}{2}; 1+\alpha}_0\left([0,T]\times\overline{\Phi(\Omega_{\omega,\delta}^+)}\right) \times 
	C^{\frac{1+\alpha}{2}; 1+\alpha}_0\left([0,T]\times\overline{\Phi(\Omega_{\omega,\delta}^-)}\right)$, 
it follows that  the pair $(v^+, v^-)$ 
defined by 
\begin{align*}
&v^\pm := \hat v^\pm+ \tilde v^\pm-u^\pm \qquad \mbox{ in }[0,T]\times\overline{\Phi(\Omega_{\omega,\delta}^\pm)}, 
\end{align*}
belongs to $C^{\frac{1+\alpha}{2}; 1+\alpha}_0\left([0,T]\times\overline{\Phi(\Omega_{\omega,\delta}^+)}\right) \times 
	C^{\frac{1+\alpha}{2}; 1+\alpha}_0\left([0,T]\times\overline{\Phi(\Omega_{\omega,\delta}^-)}\right)$
and solves problem (\ref{auxtranpb2}). 
Thus, in order to show the existence of a solution of problem $(\ref{auxtranpb1})$, it remains to show the existence of a solution  $({\hat v}^+,{\hat v}^-) $ of 
problem  $(\ref{auxtranpb3})$.
It is classical that there exists a solution 
$\hat v \in C^{\frac{1+\alpha}{2}; 1+\alpha}_0\left([0,T]\times\overline{\Phi(\Omega_{\omega,\delta})}\right)$ of the Dirichlet problem
\begin{equation*}
\begin{cases}
\partial_t  \hat v - \Delta  \hat v = 0 & \mbox{ in }]0,T\mathclose[\times \Phi(\Omega_{\omega,\delta}), \\
\hat v= u^+ & \mbox{ on }[0,T]\times \Phi(\partial\Omega_{\omega,\delta}^+ \setminus\partial\Omega),\\
\hat v=u^-  & \mbox{ on }[0,T]\times \Phi(\partial\Omega_{\omega,\delta}^- \setminus\partial\Omega),\\
\hat v(0,\cdot) =0  & \mbox{ in }  \overline{\Phi(\Omega_{\omega,\delta})}.
\end{cases}
\end{equation*}
A proof of this fact can be found for example in Baderko 
\cite[Thm 5.1]{Ba97}, Lunardi and Vespri \cite[Thm. 4.3]{LuVe91}, and 
Lieberman \cite[Thm. 6.48]{Li96}.
Then, if we set 
\begin{align*}
&\hat v^\pm := \hat v_{\big|[0,T]\times\overline{\Phi(\Omega_{\omega,\delta}^\pm)}},
\end{align*}
we have that the pair $(\hat v^+,\hat v^-)$ belongs to  $C^{\frac{1+\alpha}{2}; 1+\alpha}_0\left([0,T]\times\overline{\Phi(\Omega_{\omega,\delta}^+)}\right) \times 
	C^{\frac{1+\alpha}{2}; 1+\alpha}_0\Big([0,T]\times\overline{\Phi(\Omega_{\omega,\delta}^-)}\Big)$ 
and solves problem $(\ref{auxtranpb3})$. Thus the existence of a solution of \eqref{auxtranpb1} is proved.

Now we turn to the proof of the uniqueness of the solution of \eqref{auxtranpb1}. Since we are dealing with a linear problem, it suffices to show that,  whenever we take homogeneous data in the right-hand sides, the only solution is $(v^+,v^-)=(0,0)$. So let  $(f_0^+,f_0^-, f_1^+, f_1^-,g, g_1, h^+, h^-) = (0,0,0,0, 0,0,0,0)$ and let 
\[
(v^+, v^-) \in C^{\frac{1+\alpha}{2}; 1+\alpha}_0\left([0,T]\times\overline{\Phi(\Omega_{\omega,\delta}^+)}\right) \times 
	C^{\frac{1+\alpha}{2}; 1+\alpha}_0\left([0,T]\times\overline{\Phi(\Omega_{\omega,\delta}^-)}\right)
\]
be a solution of problem \eqref{auxtranpb1}. Let 
\begin{align*}
e^\pm(t) := \int_{\Phi(\Omega_{\omega,\delta}^\pm)}\left(v^\pm(t,y)\right)^2\,dy  \qquad\forall t \in [0,T].
\end{align*}
Since $v^\pm$ are uniformly continuous in $[0,T]\times\overline{\Phi(\Omega_{\omega,\delta}^\pm)}$, we can see that 
 $e^\pm \in C^0\big([0,T]\big)$. But the functions $e^\pm$ are actually more regular than that. Indeed, an argument based on classical differentiation theorems for integrals depending on a parameter and on a specific approximation of the support of integration (see Verchota \cite[Thm.~1.12, p.~581]{Ve84})   shows that 
 $e^\pm \in C^1\big(]0,T[\big)$. A detailed proof of this fact can be found in \cite[Lemma 5 and Prop.~2]{Lu18}.  Following the argument in the same reference we can prove that
\[
\frac{d}{dt}e^+(t)  = -2\int_{\Phi(\Omega_{\omega,\delta}^+)}|Dv^+(t,y)|^2\,dy +2\int_{\partial\Phi(\Omega_{\omega,\delta}^+)} v^+(t,y)\frac{\partial}{\partial \nu_{\Phi(\Omega_{\omega,\delta}^+)}(y)}v^+(t,y)\,d\sigma_y\,.
\]
Then we have
\begin{align*}
\frac{d}{dt}e^+(t) 
=& -2\int_{\Phi(\Omega_{\omega,\delta}^+)}|Dv^+(t,y)|^2\,dy 
	+2\int_{\Phi(\partial\Omega_{\omega,\delta}^+\setminus\partial\Omega)} v^+(t,y)\frac{\partial}{\partial \nu_{\Phi(\Omega_{\omega,\delta}^+)}(y)}v^+(t,y)\,d\sigma_y \\
	&+2\int_{\Phi(\partial\Omega)} v^+(t,y)\frac{\partial}{\partial \nu_{\Phi(\Omega_{\omega,\delta}^+)}(y)}v^+(t,y)\,d\sigma_y\\
=& -2\int_{\Phi(\Omega_{\omega,\delta}^+)}|Dv^+(t,y)|^2\,dy +2\int_{\Phi(\partial\Omega)} v^+(t,y)\frac{\partial}{\partial \nu_{\Phi_{|\partial\Omega}}(y)}v^+(t,y)\,d\sigma_y,
\end{align*}
for all $t \in \mathopen ]0,T[$. Indeed the boundary condition on $\Phi(\partial\Omega_{\omega,\delta}^+\setminus\partial\Omega)$ implies that
\[
\int_{\Phi(\partial\Omega_{\omega,\delta}^+\setminus\partial\Omega)} v^+(t,y)\frac{\partial}{\partial \nu_{\Phi(\Omega_{\omega,\delta}^+)}(y)}v^+(t,y)\,d\sigma_y=0 
\qquad \forall t \in \mathopen ]0,T[.
\]
Similarly 
\begin{align*}
\frac{d}{dt}e^-(t)
  =& -2\int_{\Phi(\Omega_{\omega,\delta}^-)}|Dv^-(t,y)|^2\,dy +2\int_{\partial\Phi(\Omega_{\omega,\delta}^-)} v^-(t,y)\frac{\partial}{\partial \nu_{\Phi(\Omega_{\omega,\delta}^-)}(y)}v^-(t,y)\,d\sigma_y \\
=& -2\int_{\Phi(\Omega_{\omega,\delta}^-)}|Dv^-(t,y)|^2\,dy 
	+2\int_{\Phi(\partial\Omega_{\omega,\delta}^-\setminus\partial\Omega)} v^-(t,y)\frac{\partial}{\partial \nu_{\Phi(\Omega_{\omega,\delta}^-)}(y)}v^-(t,y)\,d\sigma_y \\
	&+2\int_{\Phi(\partial\Omega)} v^-(t,y)\frac{\partial}{\partial \nu_{\Phi(\Omega_{\omega,\delta}^-)}(y)}v^-(t,y)\,d\sigma_y\\
=& -2\int_{\Phi(\Omega_{\omega,\delta}^-)}|Dv^-(t,y)|^2\,dy -2\int_{\Phi(\partial\Omega)} v^-(t,y)\frac{\partial}{\partial \nu_{\Phi_{|\partial\Omega}}(y)}v^-(t,y)\,d\sigma_y,
\end{align*}
for all $t \in \mathopen ]0,T[$. Then, if we set $e := e^++e^-$ in $[0,T]$, by the transmission  conditions we get
\begin{align*}
\frac{d}{dt}e(t) =&  -2\int_{\Phi(\Omega_{\omega,\delta}^+)}|Dv^+(t,y)|^2\,dy -2\int_{\Phi(\Omega_{\omega,\delta}^-)}|Dv^-(t,y)|^2\,dy \\
	&+2\int_{\Phi(\partial\Omega)} v^+(t,y)\frac{\partial}{\partial \nu_{\Phi_{|\partial\Omega}}(y)}v^+(t,y)\,d\sigma_y-2\int_{\Phi(\partial\Omega)} v^-(t,y)\frac{\partial}{\partial \nu_{\Phi_{|\partial\Omega}}(y)}v^-(t,y)\,d\sigma_y\\
=&-2\left(\int_{\Phi(\Omega_{\omega,\delta}^+)}|Dv^+(t,y)|^2\,dy +\int_{\Phi(\Omega_{\omega,\delta}^-)}|Dv^-(t,y)|^2\,dy\right),\\
\end{align*}
for all $t \in \mathopen ]0,T[$.  Hence $\frac{d}{dt}e \leq 0$ in $]0,T[$. Since $e \geq 0$ and $e(0) = 0$, it follows that $e(t) =0$ for all $t \in [0,T]$. Then $v^+ = 0$ 
and $v^- = 0$.
\end{proof}
We can now characterize the layer potentials as specific solutions of the transmission problem   \eqref{auxtranpb1}. Indeed,
by the previous theorem and by Theorem \ref{pshp}, we can see that the pair
\begin{equation}\label{slip}
(v^+[\mu]_{|[0,T]\times\overline{\Phi(\Omega_{\omega,\delta}^+)}}, v^-[\mu]_{|[0,T]\times\overline{\Phi(\Omega_{\omega,\delta}^-)}})
\end{equation}
is the unique solution of problem (\ref{auxtranpb1}) with data 
\begin{align*}
(f_0^+, f_0^-&,f_1^+, f_1^-, g, g_1, h^+, h^-) \\
&= (0,0,0,0,0,\mu,v^+[ \mu]_{|[0,T]\times{\Phi(\partial\Omega_{\omega,\delta}^+\setminus \partial\Omega)}},v^-[\mu]_{|[0,T]\times{\Phi(\partial\Omega_{\omega,\delta}^-\setminus\partial\Omega})}).
\end{align*}
Hence, problem  (\ref{auxtranpb1}) with such data uniquely identifies the pair (\ref{slip}). Similarly, the pair 
\begin{equation*}\label{dlip}
(w^+[ \mu]_{|[0,T]\times\overline{\Phi(\Omega_{\omega,\delta}^+)}}, w^-[ \mu]_{|[0,T]\times\overline{\Phi(\Omega_{\omega,\delta}^-)}})
\end{equation*}
is the unique solution of problem (\ref{auxtranpb1}) with data 
\begin{align*}
(f_0^+, f_0^-&,f_1^+, f_1^-, g, g_1, h^+, h^-) \\
& = (0,0,0,0,-\mu,0,w^+[\mu]_{|[0,T]\times{\Phi(\partial\Omega_{\omega,\delta}^+\setminus \partial\Omega)}},w^-[\mu]_{|[0,T]\times{\Phi(\partial\Omega_{\omega,\delta}^+\setminus\partial\Omega)}}).
\end{align*}
However, to proceed with our analysis we need to characterize the the pulled-back pairs 
\begin{equation*}
(v^+[\mu]\circ\Phi^T, 
v^-[\mu]\circ\Phi^T)
\end{equation*}
and
\begin{equation*}
(w^+[\mu]\circ\Phi^T , 
w^-[\mu]\circ\Phi^T)
\end{equation*}
rather than $(v^+,v^-)$ and $(w^+,w^-)$. To do so,   we will $\Phi$-pullback  problem  \eqref{auxtranpb1} and obtain a new problem defined on the fixed domain $\Omega_{\omega,\delta}$. With the right set of data the new problem will  have the pulled-back pairs as unique solutions. 

\section{Pulling-back the problem}\label{sec:pullback}

Now our aim is to pull-back problem \eqref{auxtranpb1} to have a problem on the fixed domain $\Omega_{\omega,\delta}$.
It appears, however, that in this process we cannot keep the strong formulation of the heat operator $\partial_t -\Delta$ that we have in the first and second equations of \eqref{auxtranpb1}. 
This is because the Laplace operator $\Delta$ is a second order operator and to pull-back it we should take two derivatives of $\Phi$. We should then assume $\Phi$ at least of class $C^2$, but in our paper $\Phi$ is in $C^{1,\alpha}(\overline{ \Omega_{\omega,\delta}}, \mathbb{R}^n)$. Not even the weak formulation of the heat operator would suffice. Indeed, that would yeald to pulled-back operator 
\[
\partial_t u-\mathrm{div}\left((D\Phi)^{-1}(D\Phi)^{-\top}(Du)^{\top}\right)-\frac{1}{\left|\det D\Phi\right|}D(\left|\det D\Phi\right|)(D\Phi)^{-1}(D\Phi)^{-\top}(Du)^{\top}
\]
(see Chapko, Kress and Yoon \cite[Eq. (2.10), p. 858]{ChKrYo98}) and for the term $D(\left|\det D\Phi\right|)$ to be a function we still need $\Phi$ of class $C^2$.

Then we need to adopt a different approach.  In particular, we will place the problem in a specific quotient space. This space is introduced in the following lemma, which can be verified immediately.

\begin{lemma}\label{lem:quot}
Let $\alpha \in \mathopen ]0,1[$ and $T \in \mathopen ]0,+\infty[$. Let $\Omega$ be a bounded open connected subset of $\mathbb{R}^n$ of class $C^{1,\alpha}$.
Then the set 
\begin{align*}
\mathcal{Y}_{\Omega}  := \bigg\{
w=(w_0,w_1) &\in  \,\,C_0^{\frac{1+\alpha}{2};\alpha}\left([0,T] \times \overline{\Omega}\right)\times C^{\frac{\alpha}{2};\alpha}\left([0,T]\times \overline{\Omega},\mathbb{R}^n\right) \\
&: \int_0^T\int_{\Omega}w_0\partial_t\varphi+ (D\varphi) w_1\,dxdt =0 
\quad \forall \varphi \in \mathcal{D}\left(\mathopen]0,T\mathclose[\times\Omega\right)\bigg\}
\end{align*}
is a closed linear subspace of $C_0^{\frac{1+\alpha}{2};\alpha}\left([0,T] \times \overline{\Omega}\right)\times C^{\frac{\alpha}{2};\alpha}\left([0,T]\times \overline{\Omega},\mathbb{R}^n\right) $ and the quotient 
space
\begin{equation*}\label{lem:quot1}
\mathcal{X}_\Omega :=  \quad^{\textstyle C_0^{\frac{1+\alpha}{2};\alpha}\left([0,T] \times \overline{\Omega}\right)\times C^{\frac{\alpha}{2};\alpha}\left([0,T]\times \overline{\Omega},\mathbb{R}^n\right) }\Big/_{\textstyle \mathcal{Y}_\Omega }\,,
\end{equation*}
equipped with the quotient norm, is a Banach space. We denote by $\Pi$ the canonical projection from $C_0^{\frac{1+\alpha}{2};\alpha}\left([0,T] \times \overline{\Omega}\right)\times C^{\frac{\alpha}{2};\alpha}\left([0,T]\times \overline{\Omega},\mathbb{R}^n\right)$ to $\mathcal{X}_\Omega$.
\end{lemma}
%
In the following Lemma \ref{eqhe1} we see how we can exploit the projection $\Pi$ to transform a heat equation defined on a $\Phi$-dependent set  
$[0,T]\times \Phi(\Omega)$
into an equation on the fixed set $[0,T]\times\Omega$. To do so, we need some more notation.  If $\Omega$ is a bounded open connected subset of $\mathbb{R}^n$ of class $C^{1,\alpha}$, 
$T \in \mathopen ]0,+\infty[$, and $\alpha \in \mathopen ]0,1[$, then 
\begin{align*}
B_\Omega : (C^{1,\alpha}(\overline{\Omega},\mathbb{R}^n) \cap \mathcal{A}_{\overline{\Omega}}) \times 
C_0^{\frac{1+\alpha}{2};1+\alpha}\left([0,T]\times\overline{\Omega}\right) \to C_0^{\frac{1+\alpha}{2};\alpha}\left([0,T] \times \overline{\Omega}\right)\times C^{\frac{\alpha}{2};\alpha}\left([0,T]\times \overline{\Omega},\mathbb{R}^n\right) 
\end{align*}
is the map that takes a pair $(\Phi,u)$ to
\[
B_\Omega[\Phi,u] := \left( -\left|\det D\Phi\right| u,  (D\Phi)^{-1}(D\Phi)^{-\top}(Du)^\top \left|\det D\Phi\right|\right)
\]
and 
\begin{align*}
A_\Omega : (C^{1,\alpha}(\overline{\Omega},\mathbb{R}^n) \cap \mathcal{A}_{\overline{\Omega}}) \times 
C_0^{\frac{1+\alpha}{2};1+\alpha}\left([0,T]\times\overline{\Omega}\right) \to   \mathcal{X}_\Omega
\end{align*}
is the map that takes a pair $(\Phi,u)$ to
\begin{align*}
A_\Omega[\Phi,u] := \Pi B_\Omega[\Phi,u]\,.
\end{align*}
\begin{lemma}\label{eqhe1}
Let $\alpha \in \mathopen ]0,1[$ and $T \in \mathopen ]0,+\infty[$. Let $\Omega$ be a bounded open connected subset of $\mathbb{R}^n$ of class $C^{1,\alpha}$. Let $(\tilde f_0, \tilde f_1) \in C_0^{\frac{1+\alpha}{2};\alpha}\left([0,T] \times \overline{\Omega}\right)\times C^{\frac{\alpha}{2};\alpha}\left([0,T]\times \overline{\Omega},\mathbb{R}^n\right) $ and $(\Phi,u) \in  (C^{1,\alpha}(\overline{\Omega},\mathbb{R}^n) \cap \mathcal{A}_{\overline{\Omega}}) \times 
C_0^{\frac{1+\alpha}{2};1+\alpha}\left([0,T]\times\overline{\Omega}\right)$. Then equality
\begin{equation}\label{eqhe11}
A_\Omega[\Phi,u] = \Pi ( \tilde f_0, \tilde f_1) 
\end{equation}
holds true if and only if 
\begin{equation}\label{eqhe12}
\begin{split}
&\partial_t\left(u\circ  (\Phi^T)^{(-1)}\right)- \Delta \left(u \circ  (\Phi^T)^{(-1)}\right)\\ 
&\quad= \,\, \partial_t \left\{ \tilde f_0 \circ (\Phi^T)^{(-1)}\left|\det D(\Phi^{(-1)})\right|\right\}
+ \mathrm{div}\left\{((D\Phi) \tilde f_1)\circ (\Phi^T)^{(-1)}\left|\det D(\Phi^{(-1)})\right|\right\}
\end{split}
\end{equation}
in the sense of distributions in $\mathopen]0,T\mathclose[\times\Phi(\Omega)$. 
\end{lemma}
\begin{proof}
By the definition of $A_\Omega$ and $B_\Omega$, equation \eqref{eqhe11} is equivalent to
\[
\Pi\,\left( -\left|\det D\Phi\right| u,  (D\Phi)^{-1}(D\Phi)^{-\top}(Du)^\top \left|\det D\Phi\right|\right) = \Pi\,(\tilde f_0,\tilde  f_1),
\]
which in turn we can rewrite as
\[
\begin{split}
\int_{0}^T\int_\Omega -\left|\det D\Phi\right| u\, \partial_t\varphi +& (D\varphi)(D\Phi)^{-1}(D\Phi)^{-\top}(Du)^\top\left|\det D\Phi\right| \,dxdt \\
&=
\int_{0}^T\int_\Omega \tilde  f_0 \partial_t \varphi + (D\varphi) \tilde f_1 \,dxdt \qquad \forall \varphi \in \mathcal{D}\left(\mathopen]0,T\mathclose[\times\Omega\right).
\end{split}
\]
Then, by a change of  variables in the integrals, we obtain
\begin{equation}\label{eqhe13}
\begin{split}
&\int_{0}^T\int_{\Phi(\Omega)} -(u\circ (\Phi^T)^{(-1)})\partial_t (\varphi\circ (\Phi^T)^{(-1)}) + (D(\varphi\circ (\Phi^T)^{(-1)}))(D(u\circ (\Phi^T)^{(-1)}))^\top\,dxdt\\ 
&\quad=
\int_{0}^T\int_{\Phi(\Omega)} (\tilde f_0 \circ (\Phi^T)^{(-1)})\bigl|\det D(\Phi^{(-1)})\bigr| \partial_t (\varphi\circ (\Phi^T)^{(-1)})\\
&\qquad\qquad\qquad\qquad+(D(\varphi\circ (\Phi^T)^{(-1)})) ((D\Phi)\tilde  f_1)\circ  (\Phi^T)^{(-1)}\bigl|\det D(\Phi^{(-1)})\bigr| \,dxdt
\end{split}
\end{equation}
for all $\varphi \in \mathcal{D}\left(\mathopen]0,T\mathclose[\times\Omega\right)$. Now, if $\psi \in  \mathcal{D}(\mathopen]0,T\mathclose[\times\Phi(\Omega))$ we take a sequence $\{\varphi_j\}_j$ in $\mathcal{D}\left(\mathopen]0,T\mathclose[\times\Omega\right)$ that converges to $\psi\circ \Phi^T$ in $C^{1}\left([0,T]\times \overline{\Omega}\right)$ and we see that \eqref{eqhe13} implies that  
\[
\begin{split}
&\int_{0}^T\int_{\Phi(\Omega)}-(u\circ (\Phi^T)^{(-1)})\partial_t \psi 
+(D\psi)(D(u\circ (\Phi^T)^{(-1)}))^\top \,dxdt \\ &\qquad\qquad\qquad=
\int_{0}^T\int_{\Phi(\Omega)} (\tilde f_0 \circ (\Phi^T)^{(-1)})\bigl|\det D(\Phi^{(-1)})\bigr| \partial_t \psi\\
&\qquad\qquad\qquad\qquad\qquad\quad+(D\psi) ((D\Phi)\tilde  f_1)\circ  (\Phi^T)^{(-1)}\bigl|\det D(\Phi^{(-1)})\bigr| \,dxdt\,.
\end{split}
\]
Thus,  equation \eqref{eqhe12} holds in the sense of distributions in $]0,T[ \times \Phi(\Omega)$. 

We have shown that \eqref{eqhe11} implies \eqref{eqhe12}. The proof that \eqref{eqhe12} implies \eqref{eqhe11} is similar: we have to follow the above argument  backward.
\end{proof}
So much for the pull-back of the heat equation on a set $\Phi(\Omega)$, we can now go back to the transmission problem \eqref{auxtranpb1}, which is defined on the $\Phi$-dependent pair of sets 
$(\Phi(\Omega_{\omega,\delta}^+),\Phi(\Omega_{\omega,\delta}^-))$, and transform it into a problem on $(\Omega_{\omega,\delta}^+,\Omega_{\omega,\delta}^-)$. For the first two equations we can use Lemma \ref{eqhe1}, so it remains to see what to do with the other four equations. To shorten our notation, we set 
\begin{align*}
\mathcal{Z} := &\,\,  \mathcal{X}_{\Omega_{\omega,\delta}^+}
\times  \mathcal{X}_{\Omega_{\omega,\delta}^-}
\times C_0^{\frac{1+\alpha}{2};1+\alpha}([0,T]\times\partial\Omega) 
\times C_0^{\frac{\alpha}{2};\alpha}([0,T]\times\partial\Omega)\\
&\times C_0^{\frac{1+\alpha}{2};1+\alpha}\left([0,T]\times(\partial\Omega_{\omega,\delta}^+ \setminus \partial\Omega)\right) 
\times C_0^{\frac{1+\alpha}{2};1+\alpha}\left([0,T]\times(\partial\Omega_{\omega,\delta}^- \setminus \partial\Omega)\right), 
\end{align*}
(cf.~\eqref{lem:quot} for the definition of the quotient space $\mathcal{X}_\Omega$). Then we have the following.

\begin{theorem}\label{m1eqtp}
Let   $T \in \mathopen ]0,+\infty[$, $\alpha\in \mathopen ]0,1[$. Let $\Omega$, $\omega$ and  
$\delta_\Omega$ be as in Lemma \ref{ext1}. Let $\delta \in \mathopen ]0,\delta_\Omega[$. 
For $\Phi \in C^{1,\alpha}(\overline{ \Omega_{\omega,\delta}},\mathbb{R}^n) \cap \mathcal{A}'_{\overline{\Omega_{\omega,\delta}}}$, let $R_\Phi$ denote the map from 
\[
  {C}^{\frac{1+\alpha}{2}; 1+\alpha}_0\left([0,T]\times\overline{\Omega_{\omega,\delta}^+}\right) \times  
	{C}^{\frac{1+\alpha}{2}; 1+\alpha}_0\left([0,T]\times\overline{\Omega_{\omega,\delta}^-}\right)
\]
 to $\mathcal{Z}$ defined by
\begin{align}\label{defRm1}
&R_\Phi[U^+, U^-] := \bigg(A_{\Omega_{\omega,\delta}^+}[\Phi, U^+], A_{ \Omega_{\omega,\delta}^-}[\Phi, U^-], 
U_{|[0,T]\times\partial\Omega}^+-U_{|[0,T]\times\partial\Omega}^-, \\ \nonumber
&\qquad \qquad\qquad\qquad J_\Phi[U^+, U^-], U^+_{|[0,T]\times(\partial\Omega_{\omega,\delta}^+\setminus\partial\Omega)},  
U^-_{|[0,T]\times(\partial\Omega_{\omega,\delta}^-\setminus\partial\Omega)}\bigg),
\end{align}
where 
\[
J_\Phi[U^+, U^-] := DU^+(D\Phi)^{-1}{\bf n}[\Phi]-DU^-(D\Phi)^{-1}{\bf n}[\Phi] \qquad \mbox{ on } [0,T]\times\partial\Omega\,,
\]
and 
\[
{\bf n}[\Phi] := \frac{(D\Phi(x))^{-\top}\nu_{\Omega}(x)}{|(D\Phi(x))^{-\top}\nu_{\Omega}(x)|} \qquad \forall x \in \partial \Omega\,.
\]
Then the following statements hold.
\begin{itemize}
\item[(i)]  Let $\Phi \in C^{1,\alpha}(\overline{\Omega_{\omega,\delta}},\mathbb{R}^n) \cap \mathcal{A}'_{\overline{\Omega_{\omega,\delta}} }$. Let $(F^+, F^-, G, G_1, H^+, H^-) \in \mathcal{Z}$.  Then 
 $(U^+, U^-) \in C^{\frac{1+\alpha}{2}; 1+\alpha}_0\left([0,T]\times\overline{\Omega_{\omega,\delta}^+}\right) \times  
	C^{\frac{1+\alpha}{2}; 1+\alpha}_0\left([0,T]\times\overline{\Omega_{\omega,\delta}^-}\right)$ is a solution of  the equation 
\begin{equation}\label{m1eqtp1}
R_\Phi[U^+, U^-] = (F^+, F^-, G, G_1, H^+, H^-)
\end{equation}
 if and only if the pair 
 \[
 \left(U^+ \circ(\Phi^T)^{(-1)}, U^- \circ(\Phi^T)^{(-1)}\right) \in  C^{\frac{1+\alpha}{2}; 1+\alpha}_0\left([0,T]\times\overline{\Phi(\Omega_{\omega,\delta}^+)}\right) \times  
	C^{\frac{1+\alpha}{2}; 1+\alpha}_0\left([0,T]\times\overline{\Phi(\Omega_{\omega,\delta}^-)}\right)
 \]
is a solution of problem (\ref{auxtranpb1}) with data 
\begin{align}\label{m1eqtp2}
&f_0^\pm :=   \tilde  f^\pm_0 \circ (\Phi^T)^{(-1)}\bigl|\det D(\Phi^{(-1)})\bigr|\,, \\ \nonumber
& f_1^\pm := ((D\Phi)\tilde f^\pm_1)\circ (\Phi^T)^{(-1)}\bigl|\det D(\Phi^{(-1)})\bigr|\,,\\ \nonumber
&g := G \circ(\Phi^T)^{(-1)}, \quad g_1:= G_1 \circ(\Phi^T)^{(-1)}, \quad h^\pm := H^\pm \circ(\Phi^T)^{(-1)},
\end{align}
and with $(\tilde f^\pm_0, \tilde f_1^\pm)  \in C_0^{\frac{1+\alpha}{2};\alpha}\left([0,T] \times \overline{\Omega_{\omega,\delta}^\pm}\right)  \times  C^{\frac{\alpha}{2};\alpha}\left([0,T] \times \overline{\Omega_{\omega,\delta}^\pm},\mathbb{R}^n\right)$  such that 
$\Pi (\tilde f^\pm_0, \tilde f_1^\pm)= F^\pm$.

\item[(ii)] Let $\Phi \in C^{1,\alpha}(\overline{\Omega_{\omega,\delta}},\mathbb{R}^n) \cap \mathcal{A}'_{\overline{\Omega_{\omega,\delta}}}$. Then $R_\Phi$ is a linear homeomorphism from 
\[
C^{\frac{1+\alpha}{2}; 1+\alpha}_0\left([0,T]\times\overline{\Omega_{\omega,\delta}^+}\right) \times  
	C^{\frac{1+\alpha}{2}; 1+\alpha}_0\left([0,T]\times\overline{\Omega_{\omega,\delta}^-}\right)
\]
onto $\mathcal{Z}$.
\end{itemize}
\end{theorem}

\begin{proof}
We first consider statement (i).  Let   $(U^+, U^-) \in C^{\frac{1+\alpha}{2}; 1+\alpha}_0\left([0,T]\times\overline{\Omega_{\omega,\delta}^+}\right) \times  
C^{\frac{1+\alpha}{2}; 1+\alpha}_0\Big([0,T]\times\overline{\Omega_{\omega,\delta}^-}\Big)$ satisfy  equation 
(\ref{m1eqtp1}). Lemmas \ref{lem:quot} and  \ref{eqhe1} imply  that there exist  $( \tilde f^\pm_0, \tilde  f_1^\pm)  \in C_0^{\frac{1+\alpha}{2};\alpha}\left([0,T] \times \overline{\Omega_{\omega,\delta}^\pm}\right)  \times  C^{\frac{\alpha}{2};\alpha}\left([0,T] \times \overline{\Omega_{\omega,\delta}^\pm},\mathbb{R}^n\right)$  such that 
$\Pi (\tilde f^\pm_0,  \tilde f_1^\pm) = F^\pm$ and such that
\begin{align*}
(\partial_t- \Delta) (U^\pm \circ(\Phi^T)^{(-1)}) =&\,\,\mathrm{div}\left\{((D\Phi) \tilde f^\pm_1)\circ (\Phi^T)^{(-1)}|\det D(\Phi^{(-1)})|\right\} \\
&+ \partial_t \left\{  \tilde f^\pm_0 \circ (\Phi^T)^{(-1)}|\det D(\Phi^{(-1)})| \right\}
 \qquad\mbox{ in } \mathopen]0,T\mathclose[\times\Phi(\Omega_{\omega,\delta}^\pm).
\end{align*}
Clearly we have that 
\[
U^+\circ (\Phi^T)^{(-1)}-U^-\circ(\Phi^T)^{(-1)} = G \circ(\Phi^T)^{(-1)}  = g
 \qquad \mbox{ on }[0,T]\times \Phi(\partial\Omega),
\]
and 
\[
U^\pm\circ (\Phi^T)^{(-1)}
= H^\pm \circ (\Phi^T)^{(-1)} = h^\pm
 \qquad \mbox{ on } [0,T]\times\Phi(\partial\Omega^\pm_{\omega,\delta}\setminus \partial\Omega).
\]
Since
\[
\nu_{\Phi_{|\partial \Omega}} \circ \Phi_{|\partial\Omega} =
 \frac{(D\Phi)^{-\top}\nu_{\Omega}}{|(D\Phi)^{-\top}\nu_{\Omega}|}= {\bf n}[\Phi] 
  \qquad \mbox{ on } \partial\Omega
\] 
(see,  e.g., Lanza de Cristoforis and Rossi \cite[Lem. 4.2, p. 207]{LaRo08}),  we have that 
\begin{align*}
(DU^\pm)\circ(\Phi^T)^{(-1)}&(D\Phi)^{-1}\circ \Phi^{(-1)}{\bf n}[\Phi]\circ \Phi^{(-1)}  \\
&= (DU^\pm) \circ (\Phi^T)^{(-1)} (D\Phi^{(-1)})\nu_{\Phi_{|\partial\Omega }}\\
&= \frac{\partial}{\partial \nu_{\Phi_{|\partial\Omega}}}(U^\pm \circ(\Phi^T)^{(-1)}) 
\qquad\qquad \mbox{ on } [0,T]\times\Phi(\partial\Omega).
\end{align*}
Accordingly, equality
\[
 DU^+(D\Phi)^{-1}{\bf n}[\Phi]-DU^-(D\Phi)^{-1}{\bf n}[\Phi] = G_1 \qquad \mbox{ on } [0,T]\times\partial\Omega,
\]
implies that
\[
\frac{\partial}{\partial \nu_{|\Phi_{\partial\Omega}}}(U^+ \circ(\Phi^T)^{(-1)}) 
- \frac{\partial}{\partial \nu_{|\Phi_{\partial\Omega}}}(U^- \circ(\Phi^T)^{(-1)})
=G_1  \circ (\Phi^T)^{(-1)} = g_1 \qquad \mbox{ on } [0,T]\times\Phi(\partial\Omega).
\]
Thus, the pair  
$\left(U^+ \circ(\Phi^T)^{(-1)}, U^- \circ(\Phi^T)^{(-1)}\right)$ solves the transmission problem \eqref{auxtranpb1} with data as in  \eqref{m1eqtp2}. 
The converse can be proved by reading backward the above argument. 

We now prove statement (ii). 
We can readily see that $R_\Phi$ is linear and continuous. So,  
it suffices to show that $R_\Phi$ is bijective to deduce from the open mapping theorem  that it is a homeomorphism. 
We start by proving that it is injective.  Let $(U^+,U^-) \in 
C^{\frac{1+\alpha}{2}; 1+\alpha}_0\left([0,T]\times\overline{\Omega_{\omega,\delta}^+}\right) \times  
C^{\frac{1+\alpha}{2}; 1+\alpha}_0\left([0,T]\times\overline{\Omega_{\omega,\delta}^-}\right)$ and suppose that 
\[
R_\Phi[U^+,U^-] = (0,0,0,0,0,0).
\]
By statement (i), the pair 
 \[
 \left(U^+ \circ(\Phi^T)^{(-1)}, U^- \circ(\Phi^T)^{(-1)}\right)
 \]
solves problem (\ref{auxtranpb1}) with data $(f_0^+, f_0^-,f_1^+, f_1^-, g, g_1, h^+, h^-) = (0,0,0,0,0,0,0,0)$.  Then, Theorem 
\ref{auxtranpb} implies that $\left(U^+ \circ(\Phi^T)^{(-1)}, U^- \circ(\Phi^T)^{(-1)}\right)=(0,0)$ and 
accordingly we have $(U^+,U^-)=(0,0)$.
Now it remains to show that  $R_\Phi$ is surjective.  Let 
\[
(F^+, F^-, G, G_1, H^+, H^-) \in \mathcal{Z}.
\]
It is easy to check that  $( f_0^+, f_0^-, f_1^+, f_1^-,g, g_1, h^+, h^-)$ defined as in (\ref{m1eqtp2}) belongs to 
$\mathcal{S}_\Phi$ (cf. \eqref{def:Sphi}). Accordingly, Theorem \ref{auxtranpb} implies that there exists a pair 
\[
(v^+, v^-)\in C^{\frac{1+\alpha}{2}; 1+\alpha}_0\left([0,T]\times\overline{\Phi(\Omega_{\omega,\delta}^+)}\right) \times 
	C^{\frac{1+\alpha}{2}; 1+\alpha}_0\Big([0,T]\times\overline{\Phi(\Omega_{\omega,\delta}^-)}\Big)
\]
that solves the corresponding problem \eqref{auxtranpb1}.  Then statement (i) implies that 
\[
(U^+,U^-):=(v^+\circ\Phi^T,v^-\circ\Phi^T)
\in C^{\frac{1+\alpha}{2}; 1+\alpha}_0\left([0,T]\times\overline{\Omega_{\omega,\delta}^+}\right) \times  
	C^{\frac{1+\alpha}{2}; 1+\alpha}_0\left([0,T]\times\overline{\Omega_{\omega,\delta}^-}\right)
\]
is a solution of
\[
R_\Phi[U^+,U^-]  = (F^+, F^-, G, G_1, H^+, H^-).
\]
\end{proof}
We are now ready to characterize the pairs
\begin{align*}
 \left(V^+_\Phi[\mu],  V^-_\Phi[\mu]\right):=\left(v^+[\mu\circ(\Phi^T)^{(-1)}]\circ \Phi^T , 
v^-[\mu\circ(\Phi^T)^{(-1)}]\circ\Phi^T \right)
\end{align*}
and
\begin{align*}
 \left(W^+_\Phi[\mu],  W^-_\Phi[\mu]\right):=\left(w^+[\mu\circ(\Phi^T)^{(-1)}]\circ \Phi^T , 
w^-[\mu\circ(\Phi^T)^{(-1)}]\circ\Phi^T \right)
\end{align*}
as solutions of $\Phi$-dependent equations defined on a fixed domain. This will be done in the following Theorem \ref{eqabspbm1}.  The proof is a straightforward consequence 
of Theorem \ref{m1eqtp}, of Theorems 
\ref{pdhp} and \ref{pshp} on the properties of the layer heat potentials,  and of Lemma \ref{rajacon} (i) on the change of variable in integrals over $\Phi(\partial\Omega)$.
\begin{theorem}\label{eqabspbm1}
Let $T \in \mathopen ]0,+\infty[$, $\alpha\in \mathopen ]0,1[$. Let $\Omega$, $\omega$ and  
$\delta_\Omega$ be as in Lemma \ref{ext1}. Let $\delta \in \mathopen ]0,\delta_\Omega[$. 
 Let $\Phi \in C^{1,\alpha}(\overline{\Omega_{\omega,\delta}},\mathbb{R}^n) \cap \mathcal{A}'_{\overline{\Omega_{\omega,\delta}}}$. Then the following statements hold.
\begin{itemize}
\item[(i)] Let $\mu \in C^{\frac{\alpha}{2};\alpha}_0([0,T]\times\partial\Omega)$. Then
\begin{equation}\label{eqabspbm11}
(V^+_\Phi[\mu],  V^-_\Phi[\mu]) = R_\Phi^{(-1)}(0,0,0,\mu,\mathcal{V}^+_{\delta,\Phi}[\mu],\mathcal{V}^-_{\delta,\Phi}[\mu]) 
\end{equation}
with 
\begin{align*}
\mathcal{V}^\pm_{\delta,\Phi}[\mu](t,x) := \int_0^t \int_{\partial\Omega} S_n(t-\tau,\Phi(x) - \Phi(y)&)\mu(\tau,y) \tilde\sigma_n[\Phi_{|\partial\Omega}](y)\,d\sigma_yd\tau \\ 
& \forall (t,x) \in [0,T]\times\big(\partial\Omega^\pm_{\omega,\delta}\setminus \partial \Omega\big)\,.
\end{align*}
\item[(ii)] Let $\mu \in C^{\frac{1+\alpha}{2};1+\alpha}_0([0,T]\times\partial\Omega)$. Then 
\begin{equation}\label{eqabspbm12}
(W^+_\Phi[\mu],  W^-_\Phi[\mu])=R_\Phi^{(-1)}(0,0,-\mu,0,\mathcal{W}^+_{\delta,\Phi}[\mu],\mathcal{W}^-_{\delta,\Phi}[\mu]) 
\end{equation}
with 
\begin{align*}
\mathcal{W}^\pm_{\delta,\Phi}[\mu] := - \int_0^t \int_{\partial\Omega} DS_n(t-\tau,\Phi(x) - \Phi(y))&\nu_{\Phi_{|\partial\Omega}}(\Phi(y))\mu(\tau,y) \tilde\sigma_n[\Phi_{|\partial\Omega}](y)\,d\sigma_yd\tau \\
& \forall (t,x) \in [0,T]\times\big(\partial\Omega^\pm_{\omega,\delta}\setminus \partial \Omega\big)\,.
\end{align*}
\end{itemize}
\end{theorem}

\section{Dependence of the heat layer potentials upon shape perturbations} \label{sec:main}

In this section we prove our main result.  That is, we prove that the maps $V_\phi$, $V_{l,\phi}$, $W_{*,\phi}$, and $W_\phi$  in 
\eqref{Vdef}--\eqref{Wdef} depend smoothly on the 
shape parameter $\phi$. To do so, we first use equalities \eqref{eqabspbm11} and \eqref{eqabspbm12} to show that the maps $\Phi\mapsto V^\pm_\Phi$ and $\Phi\mapsto W^\pm_\Phi$ are of class $C^\infty$. So we have to understand the regularity of the terms appearing in 
these equations, and we begin with   the map that takes $\Phi$ to $R_\Phi^{(-1)}$.

\begin{proposition}\label{Rphi}
Let  $T \in \mathopen ]0,+\infty[$, $\alpha\in \mathopen ]0,1[$. Let $\Omega$, $\omega$ and  
$\delta_\Omega$ be as in Lemma \ref{ext1}. Let $\delta \in \mathopen ]0,\delta_\Omega[$. The map that takes $\Phi\in C^{1,\alpha}(\overline{\Omega_{\omega,\delta}},\mathbb{R}^n) \cap \mathcal{A}'_{\overline{\Omega_{\omega,\delta}}}$ to 
\[
R_\Phi^{(-1)}\in\mathcal{L}\left(\mathcal{Z}\,,\;C^{\frac{1+\alpha}{2}; 1+\alpha}_0\left([0,T]\times\overline{\Omega_{\omega,\delta}^+}\right) \times  
	C^{\frac{1+\alpha}{2}; 1+\alpha}_0\left([0,T]\times\overline{\Omega_{\omega,\delta}^-}\right)\right)
\]
is real analytic. 
\end{proposition}
\begin{proof} By the definition of $R_\Phi$ in \eqref{defRm1}, to prove that 
the map that takes $\Phi\in C^{1,\alpha}(\overline{\Omega_{\omega,\delta}},\mathbb{R}^n) \cap \mathcal{A}'_{\overline{\Omega_{\omega,\delta}}}$ to  
\[
R_\Phi\in\mathcal{L}\left(C^{\frac{1+\alpha}{2}; 1+\alpha}_0\left([0,T]\times\overline{\Omega_{\omega,\delta}^+}\right) \times  
	C^{\frac{1+\alpha}{2}; 1+\alpha}_0\left([0,T]\times\overline{\Omega_{\omega,\delta}^-}\right)\,,\;\mathcal{Z}\right)
\]
is real analytic it suffices to check that the maps that take  $\Phi$ to  
\[
A_{\Omega_{\omega,\delta}^\pm}[\Phi,\cdot]\in \mathcal{L}\left(C^{\frac{1+\alpha}{2}; 1+\alpha}_0\left([0,T]\times\overline{\Omega_{\omega,\delta}^\pm}\right)\,,\;\mathcal{X}_{\Omega_{\omega,\delta}^\pm}\right)
\] 
and to 
\[
J_\Phi\in \mathcal{L}\left(C^{\frac{1+\alpha}{2}; 1+\alpha}_0\left([0,T]\times\overline{\Omega_{\omega,\delta}^+}\right) \times  
	C^{\frac{1+\alpha}{2}; 1+\alpha}_0\left([0,T]\times\overline{\Omega_{\omega,\delta}^-}\right)\,,\;C_0^{\frac{\alpha}{2};\alpha}([0,T]\times\partial\Omega)\right)
\] 
are real analytic.  This follows from the real analyticity of the map that takes an invertible matrix with 
Schauder entries to its inverse (cf., e.g., \cite[Lemma 2.1]{LaRo04}, see also \cite{La91}) and from the real analyticity of the map that takes $\Phi$ to $\mathbf{n}[\Phi]=\nu_{\Phi_{|\partial \Omega}} \circ \Phi_{|\partial\Omega}$ (cf.~Lemma \ref{rajacon}). Then, to complete the proof of the proposition it suffices to remember that the set of invertible operators is open, the map that takes an invertible operator to its inverse is real analytic, and the composition of real analytic maps is real analytic (see, e.g., Hille and Phillips \cite[Theorems 4.3.2 and 4.3.4]{HiPh57} and Prodi and Ambrosetti \cite[Theorem 11.1]{PrAm73}). 
\end{proof}

We now turn to the maps $\Phi\mapsto\mathcal{V}^\pm_{\delta,\Phi}$ and $\Phi\mapsto\mathcal{W}^\pm_{\delta,\Phi}$ of Theorem 
 \ref{eqabspbm1}. 
\begin{proposition}\label{srslhp}
Let  $T \in \mathopen ]0,+\infty[$, $\alpha\in \mathopen ]0,1[$. Let $\Omega$, $\omega$ and  
$\delta_\Omega$ be as in Lemma \ref{ext1}. Let $\delta \in \mathopen ]0,\delta_\Omega[$. 
 Then the maps that take $\Phi\in C^{1,\alpha}(\overline{\Omega_{\omega,\delta}},\mathbb{R}^n) \cap \mathcal{A}'_{\overline{\Omega_{\omega,\delta}}}$ to
\[
\mathcal{V}^\pm_{\delta,\Phi}\in \mathcal{L}\left(C^{\frac{\alpha}{2};\alpha}_0\left([0,T]\times\partial\Omega\right)\,,\; C^{\frac{1+\alpha}{2};1+\alpha}_0\left([0,T]\times\big(\partial\Omega^\pm_{\omega,\delta}\setminus \partial \Omega\big)\right)\right)
\]
and to
\[
\mathcal{W}^\pm_{\delta,\Phi}\in \mathcal{L}\left(C^{\frac{1+\alpha}{2};1+\alpha}_0\left([0,T]\times\partial\Omega\right)\,,\; C^{\frac{1+\alpha}{2};1+\alpha}_0\left([0,T]\times\big(\partial\Omega^\pm_{\omega,\delta}\setminus \partial \Omega\big)\right)\right)
\]
are of class $C^\infty$.
\end{proposition}
\begin{proof}
We  verify the statement  for $\mathcal{V}^+_{\delta,\Phi}$. 
The map from $C^{1,\alpha}(\overline{\Omega_{\omega,\delta}},\mathbb{R}^n) \cap \mathcal{A}'_{\overline{\Omega_{\omega,\delta}}}$ to $C^{1,\alpha}((\partial\Omega^+_{\omega,\delta}\setminus \partial \Omega) \times \partial \Omega, \mathbb{R}^n \setminus\{0\})$ that takes $\Phi$ to the map $\Psi$ defined by 
\[
\Psi(x,y) := \Phi(x) - \Phi(y)  \qquad \forall (x,y) \in (\partial\Omega^+_{\omega,\delta}\setminus \partial \Omega) \times \partial \Omega,
\]
is linear and continuous and therefore of class $C^\infty$. Since the composition of 
two $C^\infty$ maps is of class $C^\infty$, Lemma \ref{tdso} of the Appendix on the regularity of a superposition operator implies that the map from 
$C^{1,\alpha}(\overline{\Omega_{\omega,\delta}},\mathbb{R}^n) \cap \mathcal{A}'_{\overline{\Omega_{\omega,\delta}}}$ 
to $C^{\frac{1+\alpha}{2};1+\alpha}_0([0,T]\times((\partial\Omega^+_{\omega,\delta}\setminus \partial \Omega) \times \partial\Omega))$ that takes 
$\Phi$ to the function defined by 
\[
S_n(t,\Phi(x) -\Phi(y)) \qquad \forall (t,x,y) \in [0,T]\times((\partial\Omega^+_{\omega,\delta}\setminus \partial \Omega) \times \partial\Omega)
\]
is of class $C^\infty$.  Finally, Lemma \ref{rajacon} on the real analyticity of $\tilde \sigma_n[\cdot]$ and Lemma 
\ref{tdio} of the Appendix on the linearity and continuity of a time dependent integral operator, imply the validity of the statement for $\mathcal{V}^+_{\delta,\Phi}$. The proof of the statement for $\mathcal{V}^-_{\delta,\Phi}$ and for $\mathcal{W}^\pm_{\delta,\Phi}$ are very similar and therefore omitted.
\end{proof}

We observe that, in the proof of Proposition \ref{srslhp}, it is the regularity of the fundamental solution $S_n$ to prevent $\Phi\mapsto\mathcal{V}^\pm_{\delta,\Phi}$ and $\Phi\mapsto\mathcal{W}^\pm_{\delta,\Phi}$ from being real analytic. Indeed, for $\xi\neq 0$ the function $t\mapsto S_n(t,\xi)$ belongs to $C^\infty(\mathbb{R})$ but is not real analytic. 

We can now go back to the maps $\Phi\mapsto V^\pm_\Phi$ and $\Phi\mapsto W^\pm_\Phi$ and prove that they are smooth.

\begin{theorem}\label{sslhp}
Let  $T \in \mathopen ]0,+\infty[$, $\alpha\in \mathopen ]0,1[$. Let $\Omega$, $\omega$ and  
$\delta_\Omega$ be as in Lemma \ref{ext1}. Let $\delta \in \mathopen ]0,\delta_\Omega[$.  Then the maps that take $\Phi\in C^{1,\alpha}(\overline{ \Omega_{\omega,\delta}},\mathbb{R}^n) \cap \mathcal{A}'_{\overline{\Omega_{\omega,\delta}}}$ to 
 \[
 V^\pm_\Phi\in\mathcal{L}\left(C^{\frac{\alpha}{2};\alpha}_0([0,T]\times\partial\Omega)\,,\;C^{\frac{1+\alpha}{2};1+\alpha}_0\Big([0,T]\times\overline{\Omega^\pm_{\omega,\delta}}\Big) \right)
 \]
 and to 
 \[
 W^\pm_\Phi\in\mathcal{L}\left(C^{\frac{1+\alpha}{2};1+\alpha}_0([0,T]\times\partial\Omega)\,,\;C^{\frac{1+\alpha}{2};1+\alpha}_0\Big([0,T]\times\overline{\Omega^\pm_{\omega,\delta}}\Big) \right)
 \]
 are of class $C^\infty$.
\end{theorem}

\begin{proof}
It follows from equalities \eqref{eqabspbm11} and \eqref{eqabspbm12}, from Propositions \ref{Rphi} and \ref{srslhp}, and because the composition of a real analytic map with a $C^\infty$ map is of class $C^\infty$.
\end{proof}

As a corollary of Theorem \ref{sslhp}, we are now ready to prove our main result.

\begin{theorem}\label{thm:main}
Let   $\alpha \in \mathopen ]0,1[$, $T \in \mathopen ]0,+\infty[$.  
Let $\Omega$ be a bounded open subset of $\mathbb{R}^n$ of class $C^{1,\alpha}$ such that both 
$\Omega$ and $\Omega^-$ are connected. Then the maps that take $\phi\in C^{1,\alpha}(\partial\Omega,\mathbb{R}^n) \cap \mathcal{A}_{\partial\Omega}$ to 
\[
\begin{aligned}
&V_\phi\in\mathcal{L}\left(C^{\frac{\alpha}{2};\alpha}_0([0,T]\times\partial\Omega)\,,\;C^{\frac{1+\alpha}{2};1+\alpha}_0([0,T]\times\partial\Omega)\right)\,,\\
&V_{l,\phi}\in\mathcal{L}\left(C^{\frac{\alpha}{2};\alpha}_0([0,T]\times\partial\Omega)\,,\;C^{\frac{\alpha}{2};\alpha}_0([0,T]\times\partial\Omega)\right) \quad\text{with }l\in\{1,\dots,n\}\,,\\
&W_{*,\phi}\in\mathcal{L}\left(C^{\frac{\alpha}{2};\alpha}_0([0,T]\times\partial\Omega)\,,\;C^{\frac{\alpha}{2};\alpha}_0([0,T]\times\partial\Omega)\right)\,,\\
 &W_\phi\in\mathcal{L}\left(C^{\frac{1+\alpha}{2};1+\alpha}_0([0,T]\times\partial\Omega)\,,\;C^{\frac{1+\alpha}{2};1+\alpha}_0([0,T]\times\partial\Omega)\right) 
\end{aligned}
 \]
 are of class $C^\infty$.
 \end{theorem}
 
 \begin{proof}
It clearly suffices to show that the maps in the statement are of class $C^\infty$ in a neighborhood of a function   
 $\phi_0\in C^{1,\alpha}(\partial\Omega,\mathbb{R}^n) \cap \mathcal{A}_{\partial\Omega}$. By the definitions \eqref{Vdef}--\eqref{Wdef}, by the jump relations of the layer potentials of Theorems \ref{pdhp} and \ref{pshp}, and by the 
 extension result of Lemma \ref{ext2}, there exists an open neighborhood $\mathcal{W}_0$ of $\phi_0$ in 
 $C^{1,\alpha}(\partial\Omega,\mathbb{R}^n) \cap \mathcal{A}_{\partial\Omega}$ such that 
 \begin{align*}
 V_\phi[\mu] &= v^+[\mu \circ (\phi^T)^{(-1)}] \circ \phi^T = V^+_{\mathbf{E}[\phi]}[\mu]\,,\\
 V_{l,\phi}[\mu] &= -\frac{\mathbf{n}_l[\mathbf{E}[\phi]]}{2}\mu+ \frac{\partial}{\partial x_l}(v^+[\mu \circ (\phi^T)^{(-1)}]) \circ \phi^T \\
 & =-\frac{\mathbf{n}_l[\mathbf{E}[\phi]]}{2}\mu+ ( (DV^+_{\mathbf{E}[\phi]}[\mu]) \cdot(D\mathbf{E}[\phi])^{-1})_l\,,\\
  W_{*,\phi}[\mu] &= -\frac{1}{2}\mu+ ((Dv^+[\mu \circ (\phi^T)^{(-1)}]\circ \phi^T) \nu_\phi\circ \phi \\
 & =-\frac{1}{2}\mu+ ( (DV^+_{\mathbf{E}[\phi]}[\mu]) \cdot(D\mathbf{E}[\phi])^{-1}) \cdot \mathbf{n}[\mathbf{E}[\phi]]
 \end{align*}
for all $\phi\in\mathcal{W}_0$ and all  $\mu\in  C^{\frac{\alpha}{2};\alpha}_0([0,T]\times\partial\Omega)$,  and such that  
\[
W_\phi[\mu] = \frac{1}{2}\mu + W^+_{\mathbf{E}[\phi]}[\mu]
\]
for all $\phi\in\mathcal{W}_0$ and all  $\mu \in  C^{\frac{1+\alpha}{2};1+\alpha}_0([0,T] \times \partial\Omega)$. Thus, the statement follows by Theorem \ref{sslhp}, by Lemma \ref{ext2}, and by standard calculus in Banach spaces. 
  \end{proof}
  

 \begin{appendices}
\section{}\label{appA}
In this appendix we collect a few auxiliary and technical results on the regularity of certain composition and nonlinear time-dependent integral operators.
First of all, we introduce some notation and some definitions. 
Let $ n, s \in \mathbb{N} \setminus \{0\}$, $\alpha \in \mathopen ]0,1[$, $1 \leq s \leq n$.   
We set $\mathbb{B}_s :=\{x \in \mathbb{R}^s:|x|<1\}$.
We say that a subset $M$ of $\mathbb{R}^n$ is a differential manifold (or simply a manifold) of dimension $s$ and 
of class $C^{1,\alpha}$ imbedded in 
$\mathbb{R}^n$ if, for every $P \in M$, there exist a neighborhood $W$ of $P$ in $\mathbb{R}^n$ and a parametrization 
$\psi \in C^{1,\alpha}(\overline{\mathbb{B}_s}, \mathbb{R}^n)$ such that $\psi$ is a homeomorphism 
of $\mathbb{B}_s$ onto $W \cap M$, $\psi(0) = P$, and $D \psi$ has rank $s$ at all  points of 
$\overline{ \mathbb{B}_s}$. If we further assume that $M$ is compact, then there exist
$P_1,\ldots, P_r \in M$  and parametrizations $\{\psi_i\}_{i=1,\ldots,r}$  with 
$\psi_i \in C^{1,\alpha}(\overline{\mathbb{B}_s}, \mathbb{R}^n)$ such that $\bigcup_{i=1}^r\psi_i(\mathbb{B}_s)=M$.

Let $M$ be a  compact manifold of dimension $s$ and 
of class $C^{1,\alpha}$ imbedded in $\mathbb{R}^n$. We can use the local parametrizations to 
define the Banach space $C^{1,\alpha}(M)$ (see,  e.g., Lanza de Cristoforis and Rossi \cite[p. 142]{LaRo04}). 
Similarly, we can define the  parabolic counterpart of $C^{1,\alpha}(M)$. Let $T \in \mathopen ]0,+\infty[$. 
The space $C^{\frac{1+\alpha}{2};1+\alpha}([0,T] \times M)$
is the space of functions $f$ from $[0,T] \times M$ to $\mathbb{R}$ such that 
 \[
 f \circ  \psi_i^T \in C^{\frac{1+\alpha}{2};1+\alpha}([0,T] \times\overline{\mathbb{B}_s}) \qquad \forall i=1,\ldots,r,
 \]
where $\{\psi_i\}_{i=1,\ldots,r}$ is a parametrization of $M$. On   
$C^{\frac{1+\alpha}{2};1+\alpha}([0,T] \times M)$ we define a norm by setting
\[
\|f\|_{C^{\frac{1+\alpha}{2};1+\alpha}([0,T] \times M)} 
:= \sup_{i=1,\ldots,r}
\| f \circ \psi_i^T \|_{C^{\frac{1+\alpha}{2};1+\alpha}([0,T] \times\overline{\mathbb{B}_s})} 
\qquad \forall f \in C^{\frac{1+\alpha}{2};1+\alpha}([0,T] \times M).
\]
We can  verify that, with a different finite family of parametrizations of $M$, we  obtain an 
equivalent norm. Moreover,
$C^{\frac{1+\alpha}{2};1+\alpha}([0,T] \times M)$ with the norm $\|\cdot\|_{C^{\frac{1+\alpha}{2};1+\alpha}([0,T] \times M)}$
 is a Banach space. Then
\[
C^{\frac{1+\alpha}{2};1+\alpha}_0([0,T] \times M) := \left\{  f \in C^{\frac{1+\alpha}{2};1+\alpha}([0,T] \times M) :  
 f(0,x) = 0 \,\,\,  \forall x \in  M\right\},
\]
is  a Banach subspace of $C^{\frac{1+\alpha}{2};1+\alpha}([0,T] \times M)$.
The regularity of maps with values in $C^{\frac{1+\alpha}{2};1+\alpha}([0,T] \times M)$ can be described using the local parametrization. 
More precisely, we have the following lemma, which can be proved by exploiting the definition of norm in the spaces $C^{\frac{1+\alpha}{2};1+\alpha}$.

\begin{lemma}\label{lempar}
Let $\mathcal{X}$ be a Banach space, and let $\mathcal{O}$ be an open subset of $\mathcal{X}$.  
Let    $\alpha \in \mathopen ]0,1[$.  Let $T \in \mathopen ]0,+\infty[$. Let $M$ be a compact manifold of dimension 
$1 \leq s \leq n$ of class $C^{1,\alpha}$ imbedded in $\mathbb{R}^n$.  Let $N$ be a map from 
$\mathcal{O}$ to $C^{\frac{1+\alpha}{2};1+\alpha}([0,T] \times M)$. Let $\{\psi_i\}_{i=1,\ldots,r}$ be a parametrization of $M$. 
For $i \in \{1,\ldots,r\}$, let $C_{\psi_i}$ be the composition operator from $C^{\frac{1+\alpha}{2};1+\alpha}([0,T] \times M)$ 
to $C^{\frac{1+\alpha}{2};1+\alpha}([0,T] \times \overline{\mathbb{B}_s})$ defined by 
\[
C_{\psi_i}[f] := f \circ  \psi_i^T \qquad \forall f \in C^{\frac{1+\alpha}{2};1+\alpha}([0,T] \times M).
\]
Let $h \in \mathbb{N} \cup \{\infty\}$. Then $N$ is of class $C^h$ if and only if the operator  $C_{\psi_i} \circ N$ is of class 
$C^h$ for all $i=1,\ldots,r$.
\end{lemma}

Next, we turn to  a time dependent integral operator with kernel in a parabolic Schauder space and summable 
density function. 

\begin{lemma}\label{tdio}
Let $n_1, n_2, s_1, s_2 \in \mathbb{N}$,   $1 \leq s_1 \leq n_1$, $1 \leq s_2 \leq n_2$, 
$\alpha \in \mathopen ]0,1[$.  Let $T \in \mathopen ]0,+\infty[$. Let $M_1$, $M_2$ be two compact manifolds of dimension $s_1$, $s_2$ and 
of class $C^{1,\alpha}$  imbedded in $\mathbb{R}^{n_1}$, $\mathbb{R}^{n_2}$, 
respectively. Then the bilinear map $K$ from 
$C^{\frac{1+\alpha}{2};1+\alpha}_0\big([0,T] \times (M_1\times M_2)\big) \times L^1\big([0,T] \times M_2\big)$ to 
$C^{\frac{1+\alpha}{2};1+\alpha}_0\big([0,T] \times {M_1}\big)$   defined by 
\[
K[G,f](t,x) := \int_0^t \int_{M_2} G(t-\tau, x,y)f(\tau,y)\, d\sigma_yd\tau \quad \forall (t,x) \in [0,T] \times{M_1},
\]
for all $(G,f) \in C^{\frac{1+\alpha}{2};1+\alpha}_0\big([0,T] \times (M_1\times M_2)\big) \times L^1\big([0,T] \times M_2\big)$ is continuous.
\end{lemma}

\begin{proof}
Let $\{\phi_i\}_{i=1,\ldots,r_1}$ and $\{\psi_j\}_{j=1,\ldots,r_2}$ be local parametrizations of class 
$C^{1,\alpha}$ for $M_1$ and $M_2$, 
respectively. We can suppose that $\bigcup_{i=1}^{r_1} \phi_i(\mathbb{B}_{s_1}/2)= M_1$ and 
$\bigcup_{j=1}^{r_2} \psi_j(\mathbb{B}_{s_2}/2)= M_2$. Let $\{\theta_j\}_{j=1,\ldots,r_2}$ be a partition of unity 
subordinated to the parametrization $\{\psi_j\}_{j=1,\ldots,r_2}$. Let $\pi_1$ and  $\pi_2$ be the canonical projections 
of $\mathbb{R}^{n_1} \times \mathbb{R}^{n_2}$ onto $\mathbb{R}^{n_1}$ and $\mathbb{R}^{n_2}$, 
respectively. Clearly $M_1 \times M_2$ is a compact manifold of dimension $s_1+s_2$ and  of class 
$C^{1,\alpha}$ imbedded in $\mathbb{R}^{n_1} \times \mathbb{R}^{n_2}$, and 
\[
\{(\phi_i \circ \pi_1, \psi_j \circ \pi_2)\}_{\substack{i= 1,\ldots,r_1 \\ j=1,\ldots,r_2}}
\]
is a local parametrization of maps in $C^{1,\alpha}\left(\overline{\mathbb{B}_{s_1+s_2}},\mathbb{R}^{n_1+n_2}\right)$ for $M_1\times M_2$. It suffices to show that 
there exists a constant $c>0$ such that 
\begin{align*}
\sup_{i=1,\ldots,r_1}\|K[G,f] \circ &\phi_i^T\|_{C^{\frac{1+\alpha}{2};1+\alpha}_0([0,T] \times\overline{\mathbb{B}_{s_1}})} \\
\leq& c \sup_{\substack{i=1,\ldots,r_1\\j=1,\ldots,r_2}}
\|G \circ (\phi_i \circ \pi_1, \psi_j\circ \pi_2)^T\|_{C^{\frac{1+\alpha}{2};1+\alpha}_0([0,T] \times\overline{\mathbb{B}_{s_1+s_2}})} \\
&\times \sum_{j=1}^{r_2}\int_{0}^{T}\int_{\mathbb{B}_{s_2}}|f(\tau,\psi_j(\omega))\theta_j(\psi_j(\omega))| 
|(D\psi_j^t \cdot D\psi_j)(\omega)|^{1/2}\,d\omega d\tau,
\end{align*}
for all  $(G,f) \in C^{\frac{1+\alpha}{2};1+\alpha}_0\big([0,T] \times(M_1\times M_2)\big) \times L^1\big([0,T] \times M_2\big)$. 
The above inequality follows by the equality
\begin{align*}
K&[G,f](t,\phi_i(\xi)) \\
&=\sum_{j=1}^{r_2}\int_{0}^{t}\int_{\mathbb{B}_{s_2}}G(t-\tau, \phi_i(\xi),\psi_j(\omega))f(\tau,\psi_j(\omega))\theta_j(\psi_j(\omega))|(D\psi_j^t \cdot D\psi_j)(\omega)|^{1/2}\,d\omega d\tau \\
& \hspace{10cm}\forall (t,\xi) \in [0,T] \times{\overline{\mathbb{B}_{s_1}}},
\end{align*}
that holds for all  $(G,f) \in C^{\frac{1+\alpha}{2};1+\alpha}_0\big([0,T] \times(M_1\times M_2)\big) \times L^1\big([0,T] \times M_2\big)$, by classical differentiation theorems for integrals depending on a paramenter,  and 
by the continuity of the linear map that takes a summable function to its integral.
\end{proof}

The last result of this appendix  shows a regularity result for a time dependent superposition operator. 

\begin{lemma}\label{tdso}
Let $n_1, n_2, s \in \mathbb{N}$, $1 \leq s \leq n_2$, $n_1\geq 1$, 
$\alpha \in \mathopen ]0,1[$.  Let $T \in \mathopen ]0,+\infty[$. Let $M$ be a compact manifold of dimension $s$ and of class 
$C^{1,\alpha}$ imbedded in $\mathbb{R}^{n_2}$.  Let $\Omega $ be an open subset of $\mathbb{R}^{n_1}$. 
Let $F$ be a $C^{\infty}$ function from $\mathopen]-\infty,T]\times \Omega$ to $\mathbb{R}$ such that $F(t,x)=0$ for all $(t,x) \in \mathopen ]-\infty,0]\times \Omega$. Then the set 
\[
\mathcal{O} := \{\varphi \in C^{1,\alpha}(M,\mathbb{R}^{n_1}) : \varphi(M) \subseteq \Omega\}
\]
is open in $C^{1,\alpha}(M,\mathbb{R}^{n_1})$ and the superposition operator $T_F$ of $\mathcal{O}$ to 
$C^{\frac{1+\alpha}{2};1+\alpha}_0([0,T] \times M)$ defined by 
\[
T_F[\varphi] := F \circ  \varphi^T \qquad \forall \varphi \in \mathcal{O}
\]
is of class $C^\infty$.
\end{lemma}

\begin{proof}
The set $\mathcal{O}$ is open in $C^{1,\alpha}(M,\mathbb{R}^{n_1})$ because the $C^{1,\alpha}$-norm is stronger than 
the norm of the uniform convergence.  
Let  $\{\psi_j\}_{j=1,\ldots,r}$ be a local parametrization of class 
$C^{1,\alpha}$ for $M$.  Lemma \ref{lempar} implies that, to prove the lemma, it suffices to show that the operator 
$C_{\psi_j} \circ T_F$ from $\mathcal{O}$ to $C^{\frac{1+\alpha}{2};1+\alpha}_0([0,T] \times\overline{\mathbb{B}_s})$ 
given by 
\[
C_{\psi_j} \circ T_F[\varphi] = F \circ  (\varphi \circ \psi_j)^T \qquad  \forall \varphi \in \mathcal{O}
\]
is of class $C^\infty$ for all $j = 1,\ldots,r$.
The map from $C^{1,\alpha}(M,\mathbb{R}^{n_1})$ to $C^{1,\alpha}(\overline{\mathbb{B}_s},\mathbb{R}^{n_1})$ that 
takes $\phi$ to $\phi \circ \psi_j$ is linear and continuous and then of class $C^\infty$. Accordingly, it suffices 
to prove that the superposition operator from 
\[
\mathcal{O}' := \{\phi \in C^{1,\alpha}(\overline{\mathbb{B}_s},\mathbb{R}^{n_1}) : \phi(\overline{\mathbb{B}_s}) \subseteq \Omega\}
\]
to $C^{\frac{1+\alpha}{2};1+\alpha}_0([0,T] \times\overline{\mathbb{B}_s})$ that takes $\phi$ to 
$F \circ \phi^T$ is of class $C^\infty$.  This fact is a consequence of known results on composition operators (for instance, see B\"ohme and Tomi \cite[p. 10]{BoTo73}, Henry  \cite[p. 29]{He82}, and  
Valent  \cite[Thm 4.4, p. 35]{Va88}). Indeed, exploiting \cite[Thm. 4.4, p. 35]{Va88} we obtain that 
 the superposition operator that takes $\phi$ to 
$F \circ  \phi^T$ is of class $C^\infty$ from
\[
\mathcal{O}' := \{\phi \in C^{1,\alpha}(\overline{\mathbb{B}_s},\mathbb{R}^{n_1}) : \phi(\overline{\mathbb{B}_s}) \subseteq \Omega\}
\]
to 
\[
C^{1,\alpha}_0([0,T] \times \overline{\mathbb{B}_s}) := \{f \in C^{1,\alpha}([0,T] \times \overline{\mathbb{B}_s})  : 
f(0,x) = 0 \,\,\, \forall x \in  M\}.
\]
Finally, to complete the proof we note that the embedding of  
$C^{1,\alpha}_0([0,T] \times\overline{\mathbb{B}_s}) $
 into $C^{\frac{1+\alpha}{2};1+\alpha}_0([0,T] \times\overline{\mathbb{B}_s})$ is linear and continuous.
\end{proof}

In our paper we will apply Lemma \ref{tdso} to the fundamental solution $S_n(t,x)$ of the heat equation. We observe that $S_n(t,x)$ is real analytic in $x$ for each $t$ fixed, but it is only $C^\infty$ with respect to the pair $(t,x)$ (the problem being for $t=0$). This fact prevents us from proving a real analytic result for the layer 
potentials (for example, we can not use Valent \cite[Thm. 5.2]{Va88}).

\subsection*{Acknowledgment}

The authors are members of the `Gruppo Nazionale per l'Analisi Matematica, la Probabilit\`a e le loro Applicazioni' (GNAMPA) of the `Istituto Nazionale di Alta Matematica' (INdAM). P.L. acknowledges the support of the Project BIRD191739/19 `Sensitivity analysis of partial differential equations in
the mathematical theory of electromagnetism' of the University of Padova.

\end{appendices}

\end{document}